\documentclass[12pt]{amsart}
\usepackage{amsaddr}
\usepackage{latexsym}
\usepackage{a4}
\usepackage{amssymb}
\usepackage{setspace}
\usepackage{float}
\usepackage[english]{babel}
\usepackage{url}
\usepackage{tikz}
\usetikzlibrary{positioning}
\usepackage{subcaption}
\usepackage{caption}
\usepackage{hyperref}

\newtheorem{thm}{Theorem}[section]

\newtheorem{pro}[thm]{Proposition}
\newtheorem{lem}[thm]{Lemma}

\newtheorem{cor}[thm]{Corollary}
\theoremstyle{definition}

\newtheorem{exa}[thm]{Example}

\numberwithin{equation}{section}

   \title{The lattice of monomial clones on finite fields}

\author{Sebastian Kreinecker}
\address{Sebastian Kreinecker,
Institut f\"ur Algebra,
Johannes Kepler Universit\"at Linz,\\
Altenbergerstra{\ss}e 69,
4040 Linz,
Austria}
\email{\tt kreinecker@algebra.uni-linz.ac.at}

\thanks{Supported by the Austrian Science Fund (FWF): P29931}
\date{\today}

\newcommand{\N}{\mathbb{N}}
\newcommand{\Z}{\mathbb{Z}}
\newcommand{\F}{\mathbb{F}}

\newcommand{\genMClo}[1]{\langle #1 \rangle}
\newcommand{\addCo}[1]{\mathcal{S}_q(#1)}
\newcommand{\monClones}[1]{\mathbb{M}_{#1}}
\newcommand{\ar}[1]{[#1]}
\newcommand{\genCloZ}[2]{\langle #1 \rangle_{\Z_{#2}}}
\newcommand{\maxInd}[1]{\operatorname{mI}(#1)}
\begin{document}
\bibliographystyle{amsplain}

\begin{abstract}
We investigate the lattice of clones that are generated by a set of functions that are induced on a finite field $\F$ by
monomials.
We study the atoms and coatoms of this lattice and
investigate whether this lattice contains infinite ascending chains, or infinite descending chains, or infinite antichains.
We give a connection between the lattice of these clones and semi-affine algebras.
Furthermore, we show that the sublattice of idempotent clones of this lattice is finite and
every idempotent monomial clone is principal.
\end{abstract}

\maketitle

\section{Introduction and Preliminaries}\label{sec:Intro}

Let $\N:= \{1,2,3, \ldots\}$, let $\N_0 := \N\cup\{0\}$,
let $A$ be a finite set, let $F\subseteq \bigcup \{A^{A^i}\mid i\in \N\}$ and let $n\in\N$.
We denote $F \cap A^{A^n}$ by $F^{\ar{n}}$.
Let $C \subseteq \bigcup \{A^{A^i}\mid i\in \N\}$.
$C$ is called a \emph{clone on $A$} if
it is closed under composition of functions,
i.e., if $n,m \in \N$, $f\in C^{\ar{n}}$, $g_1, \ldots, g_n \in C^{\ar{m}}$
then $f(g_1, \ldots, g_n) \in C^{\ar{m}}$,
and if $C$ contains the projections,
i.e., for $n,j\in\N$ with $j\leq n$,
$\pi_j^n\colon A^n \to A$, $\pi_j^n(x_1, \ldots, x_n):= x_j$ lies in $C$.
The characterization of clones on a two-element set by Emil Post \cite{P:PostLat}
was the beginning of the study of clones.
Already in the case of a three-element set,
there are uncountable many clones \cite{JM:EWFB}, and thus a full description seems to be hard.
Hence the investigation of clones led to the study of clones which contain, or which contain only, specific functions.
Results for clones containing only affine mappings are given in
\cite{S:COLO}, a full characterization for polynomial clones (clones containing all the constant functions) on $\Z_p \times \Z_p$ and on $\Z_{p^2}$ for any prime $p$ with the operation $+$ can be found in \cite{B:PCCM},
a description for polynomial clones on $\Z_{pq}$ with the operation $+$ for different primes $p$, $q$ is given in \cite{AM:PCOG},
and in \cite{M:PC} it is shown that the lattice of polynomial clones on a group of square-free order that contain the group operation is finite.
For general notes on clone theory we refer to \cite{PK:FUR} and \cite{S:CUA}.

In this paper we study clones of the following kind:
Let $\alpha(1), \alpha(2), \ldots \in \N_0$.
We call $g = \prod_{i\in \N} x_i^{\alpha(i)}$ a \emph{monomial}
if there is an $n \in \N$ such that for all $i\in \N$
with $i>n$ we have $\alpha(i) = 0$ and there is a $j\in \N$ such that $\alpha(j)\neq 0$.
Then we write $g$ also as $\prod_{i=1}^n x_i^{\alpha(i)}$.
We define $\maxInd{g}$ by $\max(\{i \in \N \mid \alpha(i)\neq 0 \})$ and call $\maxInd{g}$ the \emph{arity} of $g$.
Let $q=p^t$ for some prime $p$ and for some $t\in\N$, and let $\F_q := \operatorname{GF}(q)$.
Let $n\geq \maxInd{g}$.
We call the function 
$g^{\F_q^{n}}\colon \F_q^{n} \to \F_q$,
$g^{\F_q^{n}}(x_1, \ldots, x_{n}) := \prod_{i=1}^{n} x_i^{\alpha(i)}$
the \emph{$n$-ary induced function of $g$ on $\F_q$}.
We call a function which is induced by a monomial a \emph{monomial function}.
We say that $g^{\F_q^{n}}$ is \emph{idempotent} if $g^{\F_q^{n}}(x,\ldots,x) = x$ for all $x \in \F_q$, 
and call then $g$ an \emph{idempotent monomial}.
This holds if and only if $\sum_{i=1}^{\maxInd{g}} \alpha(i) \equiv_{q-1} 1$.
We call two monomials $m_1$ and $m_2$ \emph{equivalent} if 
the 
$\maxInd{m_1}$-ary induced function of $m_1$ on $\F_q$ is equal to
the $\maxInd{m_2}$-ary induced function of $m_2$ on $\F_q$.
Note that $m_1=\prod_{i=1}^n x_i^{\alpha(i)}$ is equivalent to $m_2 = \prod_{i=1}^n x_i^{\beta(i)}$
if and only if $\alpha(i)\equiv\beta(i)$ for all $1\leq i \leq n$,
where for $a,b\in\N_0$ we let $a\equiv b$ if and only if $a = b$ or $a,b>0$ and $a\equiv_{q-1} b$.
Now let $M$ be a set of monomials.
We call $M$ a \emph{monomial clone on $\F_q$}
if
$C := \{m^{\F_q^n} \mid m\in M, n\in\N, n\geq  \maxInd{m}\}$ is a clone on $\F_q$
and
$M$ is \emph{closed under equivalent monomials},
which means if $m\in M$ and some monomial $m'$ is equivalent to $m$, then $m'\in M$. 
We say that $M$ is \emph{idempotent} if all monomials of $M$ are idempotent.
For any set $M'$ of monomials, there is a least monomial clone $\genMClo{M'}$ containing $M'$,
\emph{the monomial clone on $\F_q$ generated by $M'$}.
In order to describe clones that contain only monomial functions,
we will  describe monomial clones.
Clones that contain only monomial functions have been studied in \cite{MP:MCLOCRES}, \cite{MP:TCCSM},  \cite{MP:MCSFF}, and \cite{MP:SPGMC},
where special clones (e.g.\ generated by unary or binary monomial functions) are characterized.
This was further developed in \cite{GKS:MC}
where the binary part of an idempotent monomial clone on $\F_q$ which is generated by one single binary idempotent monomial is fully described.

Let $\monClones{q}$ be the set of all monomial clones on $\F_q$.
Let $C, D$ be two monomial clones on $\F_q$.
Then we denote by $C\vee D$ the smallest monomial clone which contains $C$ and $D$,
and by $\cap$ we denote the intersection of sets.
Then $(\monClones{q}, \vee, \cap)$ is a lattice.
The smallest monomial clone $\Delta$ is the monomial clone generated by the monomial $x_1$
which induces the clone of projections,
and the largest monomial clone $\nabla$ is the monomial clone generated by $x_1x_2$,
which induces the clone generated by the field multiplication.
The set of idempotent monomial clones on $\F_q$ with $\vee$ and $\cap$ forms a sublattice of the lattice of monomial clones on $\F_q$.
Let $C$ be a monomial clone on $\F_q$ such that $C \neq \Delta$ and $C\neq \nabla$.
We call $C$ an \emph{atom} if for all monomial clones $D$ with
$D\subseteq C$ we have $D = \Delta$ or $D=C$.
We call $C$ a \emph{coatom} if for all monomial clones $D$ with
$C\subseteq D$ we have $D = \nabla$ or $D=C$.
Let $C, D$ be two monomial clones on $\F_q$.
If $C\subseteq D$ and $C\neq D$, we write $C \subset D$.
In this paper we investigate the general structure of the lattice of monomial clones on $\F_q$ for any prime power $q$:
In Section \ref{sec:FirstObs} we start our investigation on monomials and give techniques for generating monomials.
In Section \ref{sec:wholeLatSmq} we give a full description of the lattice
of monomial clones on $\F_q$ if $q \in \{2,3,4\}$.
In Section \ref{sec:conToSemi} we give a connection between monomial clones and semi-affine algebras (see \cite{S:CUA}).
In Section \ref{sec:TopBot} we  investigate the top and the bottom of the lattice of monomial clones on $\F_q$ for any prime power $q$.
In Theorem~\ref{thm:topIsoDivLat}
we see that at the top of the lattice of monomial clones on $\F_q$, there is an interval dually isomorphic to the divisor lattice of $q-1$.
Furthermore, we get a full description of the atoms in Corollary~\ref{cor:atoms} and of the coatoms in Theorem~\ref{thm:coatoms}.
In Section \ref{sec:InfChains} we prove that
$(\monClones{q}, \vee, \cap)$ is well-partially ordered (Theorem~\ref{thm:wellPO}), i.e.,\ there are no infinite antichains and no infinite descending chains of monomial clones on $\F_q$. Furthermore, we show in Theorem~\ref{thm:squFreeAscChain} that
infinite ascending chains of monomial clones on $\F_q$ exist if and only if $q-1$ is not square-free.
In Section \ref{sec:idMonCl} we show that the lattice of idempotent monomial clones on $\F_q$ is finite and that every idempotent monomial clone on $\F_q$
is \emph{principal}, i.e.\ singly generated.

\subsection{Notation for monomial clones}

Let $q$ be a prime power.
Let $C$ be a clone on $\F_q$ which contains only monomial functions, and let
$M$ be the monomial clone such that
$C = \{m^{\F_q^n} \mid m\in M, n\in\N,  n\geq  \maxInd{m}\}$.
Let $f\in C$ be induced by the monomial $m = x_{i_1}^{\alpha(1)}\cdots x_{i_n}^{\alpha(n)}$,
where $n \in \N$, $\alpha(1),\ldots, \alpha(n) \in \N$ and $i_1,\ldots,i_n \in\N$ with $i_1<i_2<\ldots <i_n$.
Then $m$ can be seen as $i_n$-ary function and
since $m$ has $n$ variables with exponents unequal to $0$,
we say that $m$ has \emph{width $n$}.
Since $x^q=x$ for all $x\in \F_q$ and $M$ is closed under equivalent monomials,
we have that $M$ contains all these monomials $x_{i_1}^{\beta(1)}\cdots x_{i_n}^{\beta(n)}$
where for all $j\leq i_n$ we have $\beta(j) \equiv_{q-1} \alpha(j)$ and $\beta(j)>0$.
By permuting the variables we have $x_{1}^{\alpha(1)}\cdots x_{n}^{\alpha(n)} \in M$.
We see that a monomial clone is uniquely determined by 
those monomials $\prod_{i=1}^n x_i^{\alpha(i)}$
with $n\in\N$
such that for all $i\leq n$, $\alpha(i) \in \{1,\ldots,q-1\}$.
Let $i, j \in \N$.
If we set a variable $x_i$ to a variable $x_j$ of a monomial,
we say that we \emph{identify the variable $x_i$ with $x_j$}.
Since $C$ is a clone and $M$ is closed under equivalence of monomials,
$M$ is closed under substitution of monomials and thus closed under identifying variables, since
a variable induces a projection.
Let $M'$ be a set of monomials.
Then $M'$ is a monomial clone on $\F_q$ if and only if 
$\{x_i \mid i\in \N\} \subseteq M'$,
$M'$ is closed under substitution of monomials and
$M'$ is closed under equivalent monomials.

\begin{exa}
 Let $q$ be a prime power.
 We have $\Delta = \genMClo{\{x_1\}}$ and $\nabla = \genMClo{\{x_1x_2\}}$.
 Let $C$ be a monomial clone on $\F_5$.
 If $m(x_1,x_2)=x_1^2x_2 \in C$, then permuting the variables yields $m_1(x_1,x_2,x_3) = m(x_2,x_3) = x_2^2x_3 \in C$.
 Then we have $m_2(x_1,x_2,x_3)=m(x_1, m_1(x_2,x_2,x_3)) = x_1^2x_2^2x_3 \in C$,
and therefore we obtain
$m_3(x_1,x_2,x_3,x_4)= m(x_1, m_2(x_2,x_3,x_4)) = x_1^2x_2^2x_3^2x_4 \in C$.
 By identifying all the variables with $x_1$ we get that $m_4(x_1) = m_3(x_1,x_1,x_1,x_1)=x_1^7 \in C$.
 Since $x_1^7$ induces the same function as $x_1^3$,
 we have $x_1^3 \in C$, and also $x_1^{11} \in C$.
\end{exa}

Since $x^q=x$ for all $x\in \F_q$, we will describe monomial clones on $\F_q$
by investigating the arithmetic properties modulo $q-1$ of the exponents of the monomials.

\subsection{Modulus Calculation}

 Let $q$ be a prime power.
 Since $x^q=x$ for all $x\in \F_q$, we calculate modulo $q-1$ in the exponents,
 but $x^0 \neq x^{q-1}$ for $x= 0$, and thus
 if an exponent $a$ of a monomial has the property that
 $a \equiv_{q-1} 0$ and $a> 0$,
 we reduce $a$ to $q-1$ and not to $0$.
 Hence, we define for $a\in \N_0$, its representative w.r.t. $\equiv$ by
 $$\overline{a}:= \begin{cases}
              q-1, &\textmd{ if } a \bmod q-1 = 0 \textmd{ and } a > 0,\\
              a \bmod q-1, &\textmd{ otherwise}.
             \end{cases}
$$
If $C$ is a monomial clone on $\F_q$,
we denote by $\overline{C}$
the set $\{\prod_{i=1}^n x_i^{\alpha(i)} \in C \mid  \forall i \leq n \colon \alpha(i) \in \{1,\ldots, q-1\}\}$.
As we already mentioned, $C$ is uniquely determined by $\overline{C}$.

\section{First observations and procedures for generating monomials}\label{sec:FirstObs}

Let $q$ be a prime power and let $C$ be a monomial clone on $\F_q$ fixed for the rest of the section.
We start with an example of a monomial clone on $\F_q$.

\begin{lem}\label{lem:firstExa}
Let $b$ be a divisor of $q-1$.
  The set $C':=\{\prod_{i=1}^n x_i^{\alpha(i)} \mid n\in\N, \forall i\leq n\colon \alpha(i)\in\N_0, \exists j\leq n\colon \alpha(j) \neq 0, \sum _{i=1}^n \alpha(i) \equiv_b 1 \}$
  is a monomial clone on $\F_q$.
\end{lem}

\begin{proof}
Obviously $\genMClo{\{x_1\}} \subseteq C'$.
Now we show that $C'$ is closed under substitution of monomials.
 Let $n\in\N$, let $m := \prod_{i=1}^n x_i^{\alpha(i)} \in C'$,
 and let $m_j := \prod_{i=1}^{n_j} x_i^{\beta_j(i)} \in C'$ for all $j\leq n$.
 Now we get
$m(m_1, \ldots, m_n) =
  \prod_{i=1}^n \prod_{i'=1}^{n_i} x_{i'}^{\alpha(i)\cdot \beta_i(i')}$,
 and we have
 that $\sum_{i=1}^n \sum_{i'=1}^{n_i} {\alpha(i)\cdot \beta_i(i')}\equiv_b \sum_{i=1}^n \alpha(i) \equiv_b 1$.
  For all $c,d \in \N$
 with $c\equiv_{q-1} d$ we have
 $c\equiv_b d$, since $b$ divides $q-1$, and thus
 if $\prod_{i=1}^n x_i^{\alpha(i)}\in C'$, then
 $\{\prod_{i=1}^n x_i^{\beta(i)}\mid \forall i\leq n\colon \overline{\beta(i)} = \overline{\alpha(i)}\} \subseteq C'$.
Hence, $C'$ is a monomial clone on $\F_q$.
\end{proof}

\begin{lem}\label{lem:cutQ1}
 Let  $\prod_{i \in I} x_i^{\alpha(i)} \in C$ with $\lvert I \rvert \geq 2$ and $\alpha(i) \neq 0$ for all $i\in I$, and let $D \subset I$ be
 such that $\sum_{i\in D} \alpha(i) \equiv_{q-1} 0$.
 Then $\prod_{i\in I\setminus D} x_i^{\alpha(i)} \in C$.
\end{lem}

\begin{proof}
 We identify all $x_i$ for $i\in D$ with an $x_j$ where $j \in I\setminus D$.
 The result holds, since $C$ is closed under identifying variables and
 $\overline{\alpha(j) + \sum_{i\in D} \alpha(i)} = \overline{\alpha(j)}$.
\end{proof}

\begin{exa}
 Let $D$ be a monomial clone on $\F_5$.
 If $x_1^3x_2^2x_3^2 \in D$, then $2+2 \equiv_4 0$ and thus by Lemma~\ref{lem:cutQ1} we have $x_1^3 \in D$.
\end{exa}

\begin{lem}\label{lem:q1Allq1}
 Let $t\in \N$, let $\alpha(1), \ldots, \alpha(t) \in \N$ and let $j\in\N$.
 $C$ contains the monomial $(\prod_{i=1}^t x_i^{\alpha(i)})(\prod_{i=t+1}^{t+j}x_i^{q-1})$ if and only if for all $n \in \N_0$ we have that $C$ contains
 $(\prod_{i=1}^t x_i^{\alpha(i)})(\prod_{i=t+1}^{t+n}x_i^{q-1})$.
\end{lem}

\begin{proof}
 The ``if''-direction obviously holds.
 We show the ``only if''-direction.
 By Lemma~\ref{lem:cutQ1}, we have $\prod_{i=1}^t x_i^{\alpha(i)} \in C$.
 Now we proceed by induction on $n \in \N$.
 For $n=1$ we have by Lemma~\ref{lem:cutQ1} that $(\prod_{i=1}^t x_i^{\alpha(i)})x_{t+1}^{q-1}\in C$, since $j\in\N$.
 Now let $n\geq 2$.
 Let $m(x_1,\ldots, x_{t+1}) := (\prod_{i=1}^t x_i^{\alpha(i)})x_{t+1}^{q-1}$.
 By the induction hypothesis we have that
 $m'(x_1, \ldots, x_{t+n-1}) := (\prod_{i=1}^t x_i^{\alpha(i)})(\prod_{i=t+1}^{t+n-1}x_i^{q-1}) \in C$.
 Now we have
 \begin{align*}
 m'(&x_1, \ldots, x_{t+n-2},m(x_{t+n-1}, x_{t+n}, \ldots, x_{t+n}))
 \\  &= (\prod_{i=1}^t x_i^{\alpha(i)})(\prod_{i=t+1}^{t+n-2}x_i^{q-1}) x_{t+n-1}^{\alpha(1)\cdot(q-1)} x_{t+n}^{\left((\sum_{i=2}^t \alpha(i)) + (q-1)\right) \cdot (q-1)}\in C,
 \end{align*}
 since
 $C$ is a monomial clone.
 Hence $\prod_{i=1}^t x_i^{\alpha(i)}\prod_{i=t+1}^{t+n}x_i^{q-1} \in C$,
 since $\overline{b\cdot(q-1)} = q-1$ for all $b\in \N$.
 This finishes the induction step.
\end{proof}

\begin{exa}
 Let $D$ be a monomial clone on $\F_5$.
 If $D$ contains $x_1^2x_2^3x_3^4$,
 then we get by Lemma~\ref{lem:q1Allq1} that $x_1^2x_2^3x_3^4 \cdots x_{2+t}^4 \in D$ for all $t\in \N_0$.
\end{exa}

\begin{lem}\label{lem:getMonWithExp}
 Let $N, n\in\N$ with $n\leq N$ and let  $f\colon \{1,\ldots, N\} \to \N$ be injective.
 If $C$ contains $\prod_{i=1}^N x_{f(i)}^{\alpha(i)}$, then
 there exists $\gamma \in \N_0$ such that $C$ contains $x_1^{\alpha(1)}\cdots x_n^{\alpha(n)} x_{n+1}^{\gamma}$.
\end{lem}

\begin{proof}
 For all $i \in \{1, \ldots, n\}$ we substitute the variable $x_i$ for the variable $x_{f(i)}$,
and for all $i \in \{n+1, \ldots, N\}$ we substitute the variable $x_{n+1}$ for~$x_{f(i)}$.
\end{proof}

\begin{lem}\label{lem:combRule}
 Let $t \in \N_0$ and let $\alpha(1), \ldots, \alpha(t) \in \N$.
  If $C$ contains $x_1\prod_{i=1}^t x_{i+1}^{\alpha(i)}$,
  then for all $n\in\N_0$ we have $x_1 \prod_{j=0}^{n-1} \prod_{i=1}^{t}  x_{j\cdot t +i+1}^{\alpha(i)} \in C$.
\end{lem}

\begin{proof}
 Let $m(x_1,\ldots, x_{t+1}) := x_1\prod_{i=1}^t x_{i+1}^{\alpha(i)}$.
 We proceed by induction on $n$.
 The statement is true for $n=0$, since $x_1 \in C$ and for $n=1$, since $m \in C$.
 Let $n>1$.
  By the induction hypothesis, we have
  $m'(x_1,\ldots, x_{(n-1)t+1}):= x_1 \prod_{j=0}^{n-2} \prod_{i=1}^{t}  x_{j\cdot t +i+1}^{\alpha(i)} \in C$.
  Now we have
  \begin{align*}
  m''(x_1, \ldots, x_{n\cdot t+1})  &:= m'(m(x_1,x_{(n-1)t +1+1}, \ldots,  x_{n\cdot t+1}), x_2, \ldots,  x_{(n-1)t+1}) \\
   & = x_1(\prod_{i=1}^t x_{(n-1)t+i+1}^{\alpha(i)})\prod_{j=0}^{n-2}  \prod_{i=1}^{t} x_{j\cdot t +i+1}^{\alpha(i)}
    = x_1 \prod_{j=0}^{n-1} \prod_{i=1}^{t}  x_{j\cdot t +i+1}^{\alpha(i)},
  \end{align*}
  which lies in $C$, since $C$ is a monomial clone.
  This finishes the induction step.
\end{proof}

We will often use Lemma~\ref{lem:combRule} in the following context:

\begin{exa}
Let $k\in \N$.
 If $x_1 \cdots x_{1+k} \in C$,
 then we get by Lemma~\ref{lem:combRule} that for all $t\in \N$ we have $x_1\cdots x_{1+t\cdot k} \in C$.
\end{exa}

\begin{lem}\label{lem:2var1x1xq}
  If $C$ contains a monomial with two times the exponent $1$, then $x_1 \cdots x_q \in C$.
\end{lem}

\begin{proof}
By Lemma~\ref{lem:getMonWithExp} we have $m= x_1 x_2 x_3^{\alpha}\in C$ for some $\alpha \in \N_0$.
Now we get by Lemma~\ref{lem:combRule} that
$x_1 (x_{2}x_{3}^{\alpha})(x_{2+2}x_{3+2}^{\alpha})  \cdots (x_{2+(q-2)2} x_{3+(q-2)2}^{\alpha}) \in C$.
By renaming the variables we get $x_1 x_2 \cdots x_{q} x_{q+1}^{\alpha} \cdots x_{2(q-1)+1}^{\alpha}\in C$.
By identifying the variables from $x_{q+2}$ to $x_{2(q-1)+1}$ with the variable $x_{q+1}$
we get 
$x_1 x_2 \cdots x_{q} x_{q+1}^{(q-1)\cdot \alpha}\in C$.
 The result follows now from Lemma~\ref{lem:cutQ1}.
\end{proof}

\begin{lem}\label{lem:oneGcdq1Is1}
 We assume that $C$ contains $x_1^{\alpha(1)}\cdots x_n^{\alpha(n)}$
 where $n\in\N$ with $n\geq 2$, $\gcd(\alpha(1),q-1) = 1$, $\alpha(1), \alpha(2)\in \N$,
 $\alpha(i) \in \N_0$ for $i>2$.
 Then there exists $\gamma \in \N_0$ such that $x_1 x_2^{\alpha(2)}\cdots x_n^{\alpha(n)}x_{n+1}^{\gamma} \in C$.
\end{lem}

\begin{proof}
If $\alpha(1)=1$ the result is obvious.
Now we assume that $\alpha(1)\neq 1$.
Let $m(x_1,\ldots x_n) := x_1^{\alpha(1)}\cdots x_n^{\alpha(n)} \in C$.
We show by induction that for all $t\in \N$,
there is a $\gamma \in \N_0$ such that
$x_1^{\alpha(1)^t}x_2^{\alpha(2)} \cdots x_n^{\alpha(n)}x_{n+1}^{\gamma} \in C$.
For $t=1$, we have $m\in C$.
Let $t>1$.
By the  induction hypothesis, we have
$m'(x_1, \ldots, x_{n+1}) = x_1^{\alpha(1)^{(t-1)}} x_2^{\alpha(2)}\cdots x_n^{\alpha(n)}x_{n+1}^{\gamma} \in C$
for some $\gamma \in C$.
Now we get
\begin{equation*}
 \begin{aligned}
  m''(x_1,\ldots, x_{n+1}) &:= m(m'(x_1, x_{n+1} \ldots, x_{n+1}),x_2, \ldots, x_{n}) \\
  &= x_1^{\alpha(1)^t}x_2^{\alpha(2)} \cdots x_n^{\alpha(n)}x_{n+1}^{\gamma'} \in C,
 \end{aligned}
\end{equation*}
 where $\gamma' = \alpha(1)\cdot (\gamma + \sum_{i=2}^k \alpha(i))$.
 This concludes the induction step.
 If $q-1 = 1$, let $t=1$ and then $\overline{\alpha(1)^t} = \overline{\alpha(1)} = 1$, since $\alpha(1) > 0$.
 If $q-1> 1$, then let $t=\phi(q-1)$, where $\phi$ denotes Euler's totient function.
 Then $\overline{\alpha(1)^t} = 1$, because $\gcd(\alpha(1), q-1) = 1$.
 Since $C$ is closed under equivalent monomials, we get $x_1x_2^{\alpha(2)} \cdots x_n^{\alpha(n)}x_{n+1}^{\gamma'} \in C$.
\end{proof}

\begin{lem}\label{lem:twoInvThenx1xq}
Let $n\in \N$ with $n\geq 2$ and let $\alpha(1), \ldots, \alpha(n) \in \N$
with $\gcd(\alpha(1),q-1) = \gcd(\alpha(2),q-1) = 1$.
    If $C$ contains $x_1^{\alpha(1)}x_2^{\alpha(2)}\cdots x_n^{\alpha(n)}$,
  then $x_1 \cdots x_q \in C$.
\end{lem}

\begin{proof}
By Lemma~\ref{lem:oneGcdq1Is1} and Lemma~\ref{lem:getMonWithExp}
we get $x_1 x_2^{\alpha(2)}x_3^{\gamma} \in C$ for some exponent $\gamma \in \N_0$.
Since $C$ is a monomial clone, we
have $x_1^{\alpha(2)} x_2x_3^{\gamma} \in C$
and thus we get by Lemma~\ref{lem:oneGcdq1Is1} and Lemma~\ref{lem:getMonWithExp} that
$x_1x_2x_3^{\delta} \in C$ for some $\delta \in \N_0$.
 The result follows now from Lemma~\ref{lem:2var1x1xq}.
\end{proof}

\begin{lem}\label{lem:x1xkxathenx1xkqa}
  Let $k\in \N$ and let $\alpha \in \{1, \ldots, q\}$.
  We assume that $x_1\cdots x_k \in C$ and $x_1^{\alpha} \in C$.
 If $k > q-\alpha$, then $x_1\cdots x_{k-(q-\alpha)} \in C$.
\end{lem}

\begin{proof}
We assume that $\alpha < q$, otherwise the claim is trivial.
Let $t := q-1-\alpha$ and let
$m(x_1,\ldots, x_k) := x_1\cdots x_k$.
Since $k-(q-\alpha) > 0$,
we have that
$m'(x_1, \ldots, x_{k-t+1}):=m(x_1, \ldots, x_{k-t}, x_{k-t+1},\ldots, x_{k-t+1}) = x_1 \cdots x_{k-t} x_{k-t+1}^{t} \in C$.
 Since $x_1^{\alpha}\in C$, we get
 $m''(x_1, \ldots, x_{k-t+1}):= m'(x_1, \ldots, x_{k-t-1}, x_{k-t}^{\alpha}, x_{k-t+1})= x_1\cdots x_{k-t-1} x_{k-t}^{\alpha} x_{k-t+1}^{t} \in C$.
 Finally, we get $m''(x_1,\ldots,x_{k-t-1},x_{k-t-1},x_{k-t-1}) = x_1\cdots x_{k-t-1}^{1+t+\alpha} \in C$.
 Then $x_1\cdots x_{k-(q-\alpha)} \in C$, since $\overline{1+t+\alpha} = 1$ and $1+ t = q-\alpha$,
 and thus $k-t-1=k-(q-\alpha)$.
\end{proof}

As a special case we get the following corollary:

\begin{cor}\label{cor:x1xqxathenx1xa}
  We assume that $C$ contains $x_1\cdots x_q$.
 Let $\alpha \in \{1,\ldots, q\}$ such that $x_1^{\alpha} \in C$.
 Then $x_1\cdots x_{\alpha} \in C$.
\end{cor}

\begin{proof}
By Lemma~\ref{lem:x1xkxathenx1xkqa} we have $x_1\cdots x_{q-(q-\alpha)} = x_1\cdots x_{\alpha} \in C$.
\end{proof}

\begin{lem}\label{lem:maxIdClone}
 $\genMClo{\{x_1 \cdots x_q\}}$ is the clone of all idempotent monomials over $\F_q$.
\end{lem}

\begin{proof}
We have $x_1 \cdots x_q = x_1 \prod_{i=1}^{q-1}x_{i+1}$.
 By Lemma~\ref{lem:combRule} we have for each $n\in\N$ that
 $$x_1\cdots x_q x_{q+1} \cdots x_{q + n(q-1)} \in  \genMClo{\{x_1 \cdots x_q\}},$$
 and thus we can generate all monomials with total degree $1$ modulo $q-1$  by identifying and permuting variables.
 Hence,  $\genMClo{\{x_1 \cdots x_q\}}$ contains all idempotent monomials over $\F_q$.
 On the other hand, the clone of all idempotent monomial contains $x_1\cdots x_q$.
\end{proof}

More generally:

\begin{lem}\label{lem:x1xkThenAllK}
 Let $k\in \N\setminus\{1\}$.
 Then $\genMClo{\{x_1 \cdots x_{k}\}}$ consists of all monomials $x_1^{\alpha(1)} \cdots x_n^{\alpha(n)}$ with
 $\sum_{i=1}^n \alpha(i) \equiv_{\gcd(k-1,q-1)} 1$.
\end{lem}

\begin{proof}
By Lemma~\ref{lem:twoInvThenx1xq} we have $x_1 \cdots x_q \in C := \genMClo{\{x_1 \cdots x_{k}\}}$.
By identifying the variables $x_3, \ldots, x_k$ with $x_2$ of the monomial $x_1 \cdots x_{k}$,
we get that $x_1 x_2^{k-1} \in C$.
By Lemma~\ref{lem:combRule} and by identifying variables, we have
for all $t\in \N$ that
$x_1x_2^{t (k-1)} \in C$.
There are $t_1, t_2 \in \Z$ such that $t_1\cdot (k-1) + t_2 \cdot (q-1) = \gcd(k-1,q-1)$,
and thus there is a $t' \in \N$
such that $\overline{t'\cdot (k-1)} = \gcd(k-1, q-1)$.
This means that $x_1x_2^{\gcd(k-1,q-1)} \in C$,
and by identifying variables we get $x_1^{1+\gcd(k-1, q-1)} \in C$.
Since $x_1\cdots x_q \in C$ and $x_1^{1+\gcd(k-1, q-1)} \in C$,
we get by Corollary~\ref{cor:x1xqxathenx1xa} that $x_1\cdots x_{1+\gcd(k-1, q-1)} \in C$.
  By Lemma~\ref{lem:combRule} we get now for all $n\in\N$ that $x_1 \cdots x_{1+n\cdot\gcd(k-1, q-1)} \in C$.
 Hence, we can generate all monomials where the sum of the exponents is congruent to $1$ modulo $\gcd(k-1, q-1)$  by identifying variables.
 \par Finally, the set of such monomials is a monomial clone by Lemma \ref{lem:firstExa}, 
 and it clearly contains $x_1\cdots x_k$.
\end{proof}

\section{The whole lattice of monomial clones if \texorpdfstring{$q\leq 4$}{q<=4}}\label{sec:wholeLatSmq}

We start our investigation on the lattice of monomial clones for the cases $\F_2$ and $\F_3$.
The following Proposition slightly generalizes the result of \cite{MP:MCSFF}
 and \cite[Corollary 3.7]{MP:MCLOCRES} for $\F_3$ where the lattice of monomial clones which are generated by a single monomial is given.
In the case of $\F_2$ these two lattices are equal and in the
case of $\F_3$ the difference of the number of members of these two lattices is one, which means, there is exactly one monomial clone on $\F_3$ that is not singly generated.

\begin{pro}\label{pro:smallMonClon}
 The lattices of monomial clones on $\F_2$ and $\F_3$
 are given in Figure \ref{fig:LRep}.

\begin{figure}[h!]
 \begin{minipage}[t]{.4\linewidth}
 \begin{center}
\begin{tikzpicture}[scale=0.7, transform shape]
  \node (max) at (0, 5.5) {$\genMClo{\{x_1x_2\}}$};
  \node (min) at (0,-0.5) {$\genMClo{\{x_1\}}$};
  \draw (min) -- (max);

\end{tikzpicture}\subcaption{Lattice of monomial clones on $\F_2$}
 \end{center}
 \end{minipage}
 \hfill
\begin{minipage}[t]{.4\linewidth}
\begin{center}
\begin{tikzpicture}[scale=0.7, transform shape]
  \node (max) at (0, 5.5)  {$\genMClo{\{x_1x_2\}}$};
  \node (a) at (-2, 4)  {$\genMClo{\{x_1x_2x_3\}}$};
  \node (b) at (2, 4){$\genMClo{\{x_1^2, x_1x_2^2\}}$};
  \node (c) at (-2,1) {$\genMClo{\{x_1x_2^2\}}$};
  \node (d) at (2, 2.5){$\genMClo{\{x_1^2x_2^2\}}$};
  \node (e) at (2,1){$\genMClo{\{x_1^2\}}$};
  \node (min) at (0,-0.5) {$\genMClo{\{x_1\}}$};

  \draw (min) -- (c) -- (a) -- (max)
	(e) -- (d) -- (b) -- (max)
	(min) -- (e)
	(c) -- (b);

\end{tikzpicture}
 \subcaption{Lattice of monomial clones on $\F_3$}\label{fig:pM1}
 \end{center}
\end{minipage}
\caption{Monomial Clones on $\F_2$ and $\F_3$}\label{fig:LRep}
\end{figure}
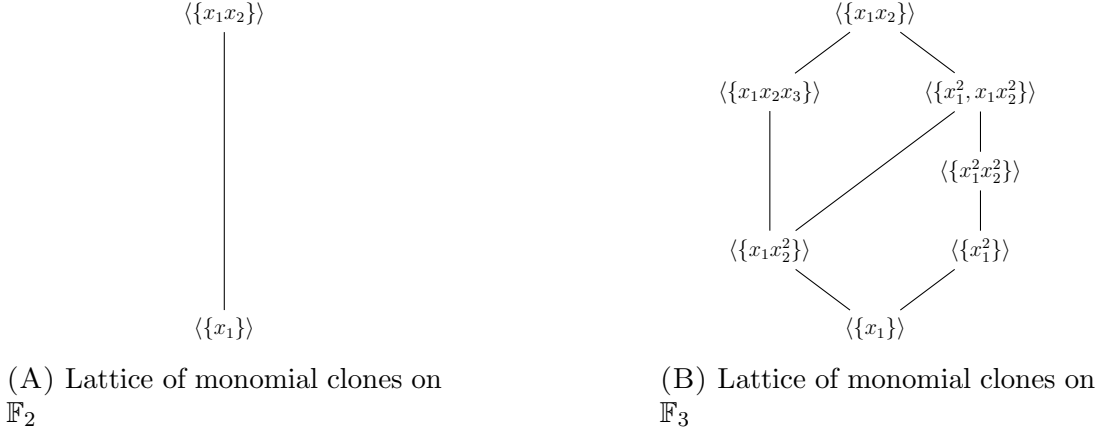

\end{pro}

\begin{proof}
Let $C$ be a monomial clone on $\F_2$.
All functions on $\F_2$ can be induced by polynomials with only $0$ and $1$ as exponents.
There is only one unary function that is induced by a monomial on $\F_2$,
namely the function induced by $x_1$.
 If $C$ contains a monomial with width $t\in \N$ with $t>1$,
 then $C$ contains $x_1\ldots x_t$.
 The result follows now from Lemma~\ref{lem:q1Allq1}.

 Let us now determine the monomial clones on $\F_3$:
  $C_1=\genMClo{\{x_1\}}$ is the smallest monomial clone.
 $C_2= \genMClo{\{x_1^2\}}$ contains all monomials of width $1$, and thus $C_1 \subset C_2$
 and there are no other monomial clones between $C_1$ and $C_2$,
 since $\overline{a} \in \{0,1,2\}$ for all $a\in\N_0$.
 Let $C_3 := \genMClo{\{x_1^2x_2^2\}}$.
 By Lemma~\ref{lem:cutQ1} and the fact that $C_2$ only contains monomials of width $1$
 we have $C_2 \subset C_3$.
 By Lemma~\ref{lem:q1Allq1} we have that $C_3$ contains all monomials $x_1^2\cdots x_n^2$ with $n\in\N$.
 By Lemma~\ref{lem:q1Allq1} every monomial in $C_3$ of width greater $1$
 generates $C_3$, and therefore there are no other monomial clones between $C_2$ and $C_3$.
 Now let $C$ be a monomial clone which contains a monomial that does not lie in $\genMClo{\{x_1\}}$
 but contains $1$ as exponent.
 Then $C$ contains $m = x_1 \prod_{i=2}^n x_i^{\alpha(i)}$ for $n\geq 2$ and $\alpha(2), \ldots, \alpha(n) \in \{1,2\}$.
 First, we assume that $1+\lvert \{j\in \{2, \ldots,n\} \mid \alpha(j) = 1 \} \rvert$ is even.
 Then by Lemma~\ref{lem:cutQ1},
 $C$ contains the monomial $x_1x_2$ and thus all monomials.
 Now we assume that $1+\lvert \{j\in \{2, \ldots,n\} \mid \alpha(j) = 1 \} \rvert$ is odd and $C$ does not contain
 a monomial with an even number ($> 0$) of $1$'s as exponent.
 If  $1+\lvert \{j\in \{2, \ldots,n\} \mid \alpha(j) = 1 \} \rvert=1$,
 we have $m= x_1 \prod_{i=2}^n x_i^{2}$.
 By Lemma~\ref{lem:cutQ1} we have $x_1x_2^{2} \in C$ and by Lemma~\ref{lem:q1Allq1}
 we have $m\in \genMClo{\{x_1x_2^{2}\}}$.
 We see that $\overline{\genMClo{\{x_1x_2^{2}\}}}$ contains exactly all monomials with exactly one exponent with $1$
 and thus $x_1^2 \not \in \genMClo{\{x_1x_2^{2}\}}$.
 If $x_1^2 \in C$, we get $x_1^2x_2^2 \in C$ by plugging $x_1^2$ into $x_1$ of $x_1x_2^2$.
 If $x_1^2 \cdots x_{n'}^2\in C$ for some $n'\geq 2$ we get by Lemma~\ref{lem:cutQ1} that
 $x_1^2 \in C$ and thus there are no monomial clones between
 $\genMClo{\{x_1x_2^2\}}$ and $\genMClo{\{x_1^2, x_1x_2^2\}}$, and neither between
  $\genMClo{\{x_1^2x_2^2\}}$ and $\genMClo{\{x_1^2, x_1x_2^2\}}$.
 Now we assume $1+\lvert \{j\in \{2, \ldots,n\} \mid \alpha(j) = 1 \} \rvert>1$.
 Then, by Lemma~\ref{lem:cutQ1}, $C$ contains the monomial $x_1x_2x_3$.
 We see that $x_1x_2x_3$ does not generate $x_1^2$,
 since $\overline{\genMClo{\{x_1x_2x_3\}}}$ contains
 all monomials with an odd number of variables with exponent $1$.
 We also have $\genMClo{\{x_1x_2^2\}} \subset \genMClo{\{x_1x_2x_3\}}$,
 since all monomials of $\overline{\genMClo{\{x_1x_2^2\}}}$ contain exactly one exponent with $1$.
 If $C$ contains $x_1x_2x_3$ and $x_1^2$, then $C$ contains $x_1x_2x_3^2$ and Lemma~\ref{lem:cutQ1} yields $x_1x_2\in C$.
 This finishes the proof.
\end{proof}

The next proposition describes the whole lattice of monomial clones if $q=4$.

\begin{pro}\label{pro:latM4}
 The lattice of monomial clones on $\F_4$ is given in Figure~\ref{fig:lat4}.
 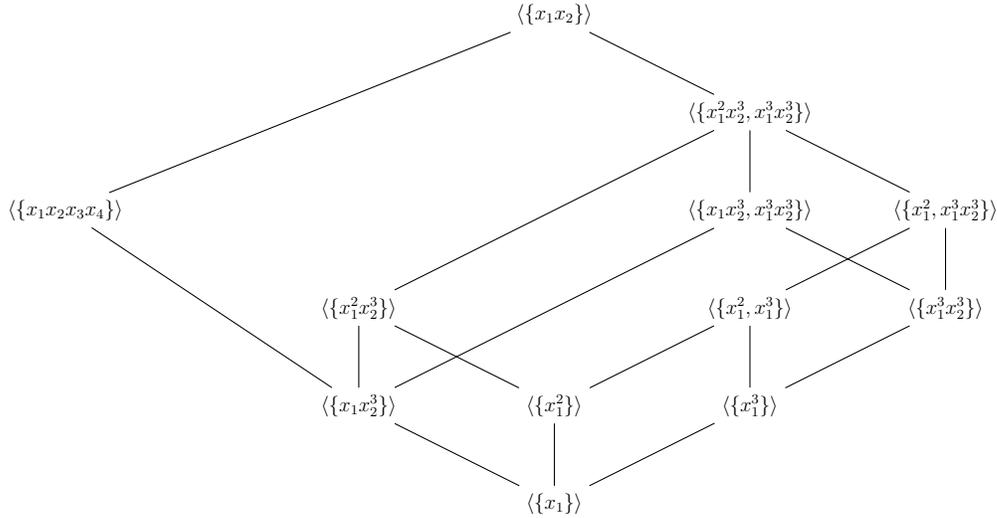
\begin{figure}[ht!]
 \begin{center}
\begin{tikzpicture}[scale=0.65, transform shape]
  \node (max) at (0, 10) {$\genMClo{\{x_1x_2\}}$};
  \node (min) at (0,0) {$\genMClo{\{x_1\}}$};
  \node (a) at (-4,2) {$\genMClo{\{x_1x_2^{3}\}}$};
  \node (b) at (0, 2) {$\genMClo{\{x_1^2\}}$};
  \node (c) at (4, 2) {$\genMClo{\{x_1^{3}\}}$};
  \node (d) at (-4, 4) {$\genMClo{\{x_1^2x_2^{3}\}}$};
  \node (e) at (4, 4){$\genMClo{\{x_1^2, x_1^{3}\}}$};
  \node (f) at (-10, 6){$\genMClo{\{x_1x_2x_3x_4\}}$};

  \node (c1) at (8, 4) {$\genMClo{\{ x_1^{3}x_2^{3}\}}$};
  \node (a1) at (4,6) {$\genMClo{\{x_1x_2^{3}, x_1^{3}x_2^{3}\}}$};
  \node (b1) at (8, 6) {$\genMClo{\{x_1^2, x_1^{3}x_2^{3}\}}$};
  \node (d1) at (4, 8) {$\genMClo{\{x_1^2x_2^{3}, x_1^{3}x_2^{3}\}}$};

  \draw (min) -- (a) -- (d)
	(b) -- (e) --(b1)
	(a)-- (a1)
	(c) -- (c1) --(b1) --(d1)
	(c1) -- (a1) --(d1)
	(min) -- (b) -- (d) --(d1) --(max)
	(a) -- (f) -- (max)
	(min) -- (c) --(e);

\end{tikzpicture}\caption{Lattice of monomial clones on $\F_4$}\label{fig:lat4}
 \end{center}
\end{figure}
\end{pro}

\begin{proof}
 Let $C$ be a monomial clone on $\F_4$.
 If $C$ contains $x_1x_2$, then $C$ contains all monomials.
 First, we show that $\genMClo{\{x_1x_2x_3x_4\}}$ and $\genMClo{\{x_1x_2\}}$
 are different and that these two monomial clones are the only monomial clones that
 contain a monomial which contains two exponents different from $0$ modulo $3$.
 To this end, we assume that $C$ contains a monomial $m= x_1^{\alpha(1)}x_2^{\alpha(2)}\cdots x_n^{\alpha(n)}$
 where $n\in\N\setminus\{1\}$ and at least 2 exponents are not equal to $0$ modulo $3$.
 We assume that $\alpha(1)\not \equiv_{3} 0$ and $\alpha(2) \not \equiv_{3} 0$.
 Then Lemma~\ref{lem:twoInvThenx1xq} yields $x_1 x_2x_3x_4\in C$,
 since $\gcd(\alpha(1), 3) = 1$ and $\gcd(\alpha(2), 3) = 1$.
 By Lemma~\ref{lem:maxIdClone} $C$ contains all idempotent monomials.
 Now we assume that $x_1\cdots x_4 \in C$
 and $C$ contains a monomial $m'$ which is not idempotent.
 By identifying all variables of $m'$ with $x_1$,
 we get that there exists $\alpha \in \N$ such that $C$ contains $x_1^{\alpha}$.
 Since $m'$ is not idempotent, we have $\overline{\alpha} \neq 1$.
 If $\overline{\alpha} = 3$,
 we have $x_1x_2x_3x_4, x_1^{3} \in C$
 and thus $x_1x_2x_3^{3}x_4^{3} \in C$.
 Now Lemma~\ref{lem:cutQ1} yields $x_1 x_2 \in C$.
 If $\overline{\alpha} = 2$,
 then $C$ contains $x_1^{2}$.
 This means that $x_1x_2x_3x_4, x_1^2 \in C$ and thus Corollary~\ref{cor:x1xqxathenx1xa}
 yields $x_1x_2 \in C$.
 If $C$ contains only idempotent monomials,
 then we get by Lemma~\ref{lem:maxIdClone} that $C = \genMClo{\{x_1x_2x_3x_4\}}$
 and thus $x_1x_2 \not \in C$.
 By the previous analysis we have,
 if $C$ contains a monomial which is not idempotent and it contains two exponents which are not $0$ modulo $3$,
 then $C$ contains $x_1x_2$.

 Now we assume that all monomials of $C$ contain at most one exponent which is not equal to $0$ modulo $3$.
 By Lemma~\ref{lem:cutQ1} and Lemma~\ref{lem:q1Allq1}
 we can restrict to the following monomials as possible generators for $C$:
 $x_1^2$, $x_1^{3}$, $x_1x_2^{3}$, $x_1^2x_2^{3}$, and $x_1^{3}x_2^{3}$.
 These monomials induce different functions.
Now we want to find all monomial clones which we get by different combinations of these generators.
We start with different combinations including $x_1^{3}x_2^{3}$.
 We have that $x_1^2 \in \genMClo{\{x_1^2x_2^{3}\}}$,
 $x_1^{3} \in \genMClo{\{x_1^{3}x_2^{3}\}}$
 and $x_1x_2^{3} \in \genMClo{\{x_1^2x_2^{3}\}}$.
 Hence, the monomial clones above $\genMClo{\{x_1^{3}x_2^{3}\}}$ and below $\genMClo{\{x_1x_2\}}$
 are given by $\genMClo{\{x_1x_2^{3}, x_1^{3}x_2^{3}\}}$,
 $\genMClo{\{x_1^2, x_1^{3}x_2^{3}\}}$, and $\genMClo{\{x_1^2x_2^{3}, x_1^{3}x_2^{3}\}}$.
 We have that $x_1^2  \not \in \genMClo{\{x_1x_2^{3}, x_1^{3}x_2^{3}\}}$,
 since $x_1^2$ does not preserve the set $\{1, a\}$ where $a$ is a generator for $(\F_4\setminus\{0\}, \cdot)$, 
 but the operations on $\F_4$ given by $(x_1,x_2) \mapsto x_1x_2^3$ and $(x_1,x_2) \mapsto x_1^3x_2^3$
 preserve the set $\{1,a\}$.
 On the other hand we have that $x_1x_2^3 \not \in \genMClo{\{x_1^2, x_1^{3}x_2^{3}\}}$,
 since all monomials of $\genMClo{\{x_1^2, x_1^{3}x_2^{3}\}}$ of width larger than $1$ have the property that all exponents are $0$ modulo $3$.
Therefore, ${\genMClo{\{x_1x_2^{3}, x_1^{3}x_2^{3}\}} \not \ni x_1^2}$ and 
 $\genMClo{\{x_1^2, x_1^{3}x_2^{3}\}}\not \ni x_1x_2^{3}$ are distinct
 and both monomial clones are strictly contained in $\genMClo{\{x_1^2x_2^{3}, x_1^{3}x_2^{3}\}}$.
 Then the combinations of the generators
 $x_1^2$, $x_1^{3}$, $x_1x_2^{3}$, $x_1^2x_2^{3}$ are left.
 Now we try to find different monomial clones if $x_1^2x_2^{3}$ is a generator.
 We have $x_1^2 \in \genMClo{\{x_1^2x_2^{3}\}}$,
 $x_1x_2^{3}\in \genMClo{\{x_1^2x_2^{3}\}}$,
 and if $x_1^{3}$ and $x_1^2x_2^{3}$ lie in $C$, then $x_1^{3}x_2^{3}\in C$.
 Hence, we just get the monomial clone $\genMClo{\{x_1^2x_2^{3}\}}$ as a new one.
 Now the combinations of the generators $x_1^2$, $x_1^{3}$, $x_1x_2^{3}$ are left.
 If $x_1^{3} \in C$ and $x_1x_2^{3} \in C$, then $x_1^{3}x_2^{3} \in C$, and
 if $x_1^{2} \in C$ and $x_1x_2^{3} \in C$, then $x_1^{2}x_2^{3} \in C$.
 Hence, we only get the monomial clones
 $\genMClo{\{x_1x_2^{3}\}}$, $\genMClo{\{x_1^2\}}$, $\genMClo{\{x_1^{3}\}}$
 and $\genMClo{\{x_1^2, x_1^{3}\}}$, which we have not found yet.
 The smallest monomial clone is given by $\genMClo{\{x_1\}}$.
 Altogether we have found $12$ different monomial clones, which are ordered as given in Figure~\ref{fig:lat4}.
\end{proof}

\section{Connection to semi-affine algebras} \label{sec:conToSemi}

In this section we give a connection of monomial clones to semi-affine algebras.
 First, we recall the definition of a semi-affine algebra (cf.~\cite{S:CUA}).
 Let $\mathbf{A}=(A,+,-,0)$ be an abelian group.
Let $n\in\N_0$ and let $f$ be an $n$-ary operation on $A$.
We call $f$ \emph{affine with respect to $\mathbf{A}$} if $f(u+'v) + f(0,\ldots, 0) = f(u) + f(v)$ for all $u, v\in A^n$,
where $+'$ is the componentwise addition.
We call then an algebra $\mathbf{B} = (A, F)$ \emph{semi-affine with respect to} $\mathbf{A}$
if every $f\in F$ is affine with respect to $\mathbf{A}$.
Furthermore, we say that $f$ is \emph{$0$-preserving} if $f(0,\ldots, 0) = 0$.
Let $F \subseteq \bigcup \{A^{A^i}\mid i\in\N\}$.
We call the algebra $(A, F)$ \emph{$0$-preserving} if
all term functions of $(A, F)$ are $0$-preserving.
Let $C$ be a clone on $A$.
We call $C$ \emph{$0$-preserving semi-affine with respect to} $\mathbf{A}$ if $(A, C)$ is a $0$-preserving and semi-affine algebra with respect to $\mathbf{A}$.
For $a\in \N_0$, we write $[a] \in \Z_{q-1}$ for the equivalence class of $a$ modulo $q-1$.
Let $q$ be a prime power and let $G$ be a set of monomials.
Then we define
$$\varphi(G) := \{f\colon \Z_{q-1}^n \to \Z_{q-1},(y_1, \ldots, y_n) \mapsto \sum_{i=1}^n [r(i)] y_i \mid n \in \N, \prod_{i =1}^n x_i^{r(i)}  \in G \}.$$
If $G$ is a monomial clone on $\F_q$,
then the set $\varphi(G)$ is a $0$-preserving semi-affine clone with respect to $(\Z_{q-1},+,-,0)$.
Note that
for all $C, D \in \monClones{q}$ with $C \subseteq D$ we have $\varphi(C) \subseteq \varphi(D)$.
On the other hand we have that
the monomial clones $\varphi(\genMClo{\{x_1^{q-1} \}})$ and $\varphi(\genMClo{\{x_1^{q-1}x_2^{q-1} \}})$
are both equal to the clone on $\Z_{q-1}$ that is generated by $\{0\}$,
but $\genMClo{\{x_1^{q-1} \}} \subset \genMClo{\{x_1^{q-1}x_2^{q-1} \}}$.
Let $F$ be a set of finitary operations on $\Z_{q-1}$.
We denote the clone generated by $F$ by $\genCloZ{F}{q-1}$.
Let us consider the difference of monomial clones and semi-affine algebras with respect to $(\Z_{q-1},+,-,0)$ in the example $q=3$.
Let $f_0\colon \Z_2 \to \Z_2, x \mapsto x$,
$f_1\colon \Z_2 \to \Z_2, x \mapsto 0$,
$f_2\colon \Z_2^3 \to \Z_2, (x,y,z) \mapsto x+y+z$,
and $f_3\colon \Z_2^2 \to \Z_2, (x, y) \mapsto x+y$.
Then the lattice of $0$-preserving semi-affine clones with respect to $(\Z_{2},+,-,0)$ is given by Figure \ref{fig:latSAA}.

 \begin{figure}[ht!]
 \begin{center}
\begin{tikzpicture}
  \node (a) at (0, 0) {$\genCloZ{\{f_0\}}{2}$};
  \node (b) at (-1.5, 1.5){$\genCloZ{\{f_2\}}{2}$};
  \node (c) at (1.5, 1.5) {$\genCloZ{\{f_1\}}{2}$};
  \node (d) at (0, 3) {$\genCloZ{\{f_3\}}{2}$};

  \draw (a)-- (b)-- (d)
  		(a) -- (c)-- (d);

\end{tikzpicture}\caption{Lattice of $0$-preserving semi-affine clones with respect to $(\Z_{2},+,-,0)$}\label{fig:latSAA}
 \end{center}
\end{figure}
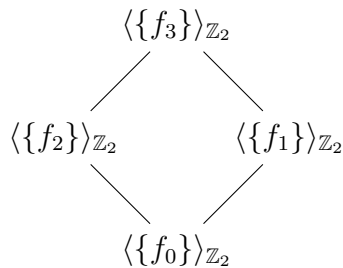

In Propositon \ref{pro:smallMonClon} we find the lattice of monomial clones on $\F_3$ and see that $\lvert \monClones{3} \rvert = 7$.
We will give a connection between
monomial clones on $\F_q$ and $0$-preserving semi-affine algebras with respect to $(\Z_{q-1},+,-,0)$ via $\varphi$ in Proposition~\ref{pro:conSemiAff},
which we use in Section~\ref{sec:InfChains} and Section~\ref{sec:idMonCl}.
If $q=4$ we have given the lattice of monomial clones in Theorem~\ref{pro:latM4}.
From this lattice we can easily derive the lattice of $0$-preserving semi-affine clones with respect to $(\Z_{3},+,-,0)$ given in Figure~\ref{fig:latSAA3}.
For this figure, we use $g_0 \colon \Z_{3}\to \Z_{3}, x \mapsto x$,
$g_1 \colon \Z_{3} \to \Z_{3}, x \mapsto 0$,
$g_2 \colon \Z_{3} \to \Z_{3}, x\mapsto 2x$,
$g_3 \colon \Z_{3}^4 \to \Z_{3}, (x_1, x_2, x_3, x_4)\mapsto x_1 +x_2 +x_3 +x_4$,
and let
$g_4\colon \Z_{3}^2 \to \Z_{3}, (x,y) \mapsto x+y$.

 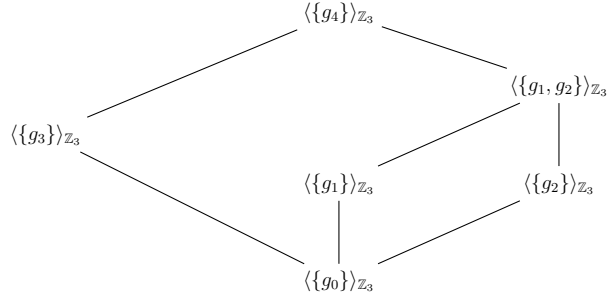
\begin{figure}[ht!]
 \begin{center}
\begin{tikzpicture}[scale=0.65, transform shape]
  \node (a) at (0, -0.5) {$\genCloZ{\{g_0\}}{3}$};
  \node (b) at (-6, 2.5){$\genCloZ{\{g_3\}}{3}$};
  \node (c) at (0, 1.5) {$\genCloZ{\{g_1\}}{3}$};
  \node (d) at (4.5, 1.5) {$\genCloZ{\{g_2\}}{3}$};
  \node (e) at (4.5, 3.5) {$\genCloZ{\{g_1,g_2\}}{3}$};
  \node (f) at (0, 5) {$\genCloZ{\{g_4\}}{3}$};

  \draw (a)-- (b)-- (f)
  		(a) -- (c)-- (e) -- (f)
  		(a) -- (d) -- (e);

\end{tikzpicture}\caption{Lattice of $0$-preserving semi-affine clones with respect to $(\Z_{3},+,-,0)$}\label{fig:latSAA3}
 \end{center}
\end{figure}

\newpage

 Now we investigate some properties of the lattice of monomial clones if $q-1$ is square-free.

\begin{lem}\label{lem:consemcon1}
 Let $q$ be a prime power such that $q-1$ is square-free, let $n\in\N_0$ and let $d, \alpha(1), \ldots, \alpha(n) \in \N$ .
 Let $C$ be a monomial clone on $\F_q$.
 If $x_1^dx_2^d\prod_{i=1}^n x_{2+i}^{\alpha(i)}\in C$,
 then $x_1^dx_2^d(\prod_{i=1}^n x_{2+i}^{\alpha(i)})x_{3+n}^{q-1}\in C$.
\end{lem}

\begin{proof}
 Let $m(x_1, \ldots, x_{n+2}) := x_1^dx_2^d\prod_{i=1}^n x_{2+i}^{\alpha(i)}$.
 We show by induction that for all $k\in \N$,
 we have
 $$m_k=(\prod_{j=1}^k x_j^{d^j}) x_{k+1}^{d^k} (\prod_{j=1}^k \prod_{i=1}^n x_{(j-1) n+i + k + 1 }^{d^{k-j} \cdot \alpha(i)}) \in C.$$
For $k=1$, we have $m_1=m \in C$.
Now we assume that $k>1$ and the induction hypothesis holds for $k-1$.
This means
$$m_{k-1}(x_1, \ldots, x_{(n+1)(k-1)+1}) = (\prod_{j=1}^{k-1} x_j^{d^j}) x_{k}^{d^{k-1}} (\prod_{j=1}^{k-1}\prod_{i=1}^n  x_{(j-1)n+i + k }^{d^{k-1-j} \cdot \alpha(i)})  \in C.$$
By substitution of monomials we obtain
\begin{align*}
 m(x_1, &m_{k-1}(x_2, \ldots, x_{(n+1)(k-1)+2}), x_{(n+1)(k-1)+3},  \ldots, x_{(n+1)k +1} ) \\
 &=
 x_1^d \left((\prod_{j=1}^{k-1} x_{j+1}^{d^j}) x_{k+1}^{d^{k-1}} (\prod_{j=1}^{k-1}\prod_{i=1}^n  x_{(j-1) n+i + k +1}^{d^{k-1-j} \cdot \alpha(i)})\right)^d (\prod_{i=1}^n x_{(n+1)(k-1)+i+2}^{\alpha(i)})\\
  &= x_1^d (\prod_{j=1}^{k-1} x_{j+1}^{d^{j+1}}) x_{k+1}^{d^{k}} (\prod_{j=1}^{k-1}\prod_{i=1}^n  x_{(j-1) n+i + k +1}^{d^{k-j} \cdot \alpha(i)}) (\prod_{i=1}^n x_{(k-1)n+i+k+1}^{\alpha(i)})\\
 &= (\prod_{j=1}^k x_j^{d^j}) x_{k+1}^{d^k} (\prod_{j=1}^k \prod_{i=1}^n x_{(j-1) n+i + k + 1 }^{d^{k-j} \cdot \alpha(i)}) \in C,
\end{align*}
 which concludes the induction step.
  Let $d' := \gcd(d, q-1)$ and let $d'' := \frac{q-1}{d'}$.
 Moreover, let $k':= \phi(d'')\cdot (q-1) + 1$, where $\phi$ denotes Euler's totient function.
 
 Case 1. We assume $n>0$:
 In $m_{k'}$ we set $x_1, \ldots, x_{k'},$ to $x_1$
 and $x_{k'+1}$ to $x_2$.
 Now we set the last $n$ variables.
 To this end, the $x_i$ with $i> (k'-1)n+k'+1$
 are set by
 $x_{(k'-1)n+k'+1+l}:= x_{2+l}$ for all $l\in \{1,\ldots, n\}$.
 For those $i$ with $k'+1<i\leq (k'-1)n+k'+1$,
 we set $x_i:= x_{3+n}$.
Hence, we get 
 \begin{align*}
  m'= x_1^{\sum_{j=1}^{k'} d^{j}} x_2^{d^{k'}} x_{3+n}^{\sum_{j=1}^{k'-1} \sum_{i=1}^n  d^{k'-j} \cdot \alpha(i)} \prod_{i=1}^n x_{2+i}^{\alpha(i)} \in C.
 \end{align*}
 Since $q-1$ is square-free,
 we have  $\gcd(d, \frac{q-1}{d'}) = 1$,
 and thus
 $d^{\phi(d'')} \equiv_{d''} 1$, since $d'' =\frac{q-1}{d'}$.
 Hence $d^{\phi(d'')+1}\equiv_{q-1} d$,
 and thus
 if $k, l \in \N$ with $k \equiv_{\phi(d'')} l$, then
 $d^k \equiv_{q-1} d^l$.
 Therefore we have $\overline{d^{k'}}=\overline{d^{\phi(d'')\cdot (q-1) + 1}} = \overline{d}$,
 $\overline{\sum_{j=1}^{k'} d^{j}} 
 = \overline{\sum_{j=1}^{\phi(d'')\cdot (q-1) + 1} d^{j}}
 = (\overline{\sum_{j=1}^{\phi(d'')\cdot (q-1)} d^{j}) + d}
 = \overline{(q-1)\sum_{j=1}^{\phi(d'')} d^{j} + d}  = \overline{d}$,
and 
 \begin{align*}\overline{\sum_{j=1}^{k'-1} \sum_{i=1}^n  d^{k'-j} \cdot \alpha(i)}   
 &=\overline{\sum_{i=1}^n \alpha(i) \sum_{j=1}^{k'-1}d^j}
 = \overline{\sum_{i=1}^n \alpha(i)(q-1) \sum_{j=1}^{\phi(d'')} d^j}\\
 &= \overline{(q-1)\underbrace{\sum_{i=1}^n \alpha(i)\sum_{j=1}^{\phi(d'')}d^j}_{\neq 0}}
 = q-1.
 \end{align*}
  Hence we have 
  $x_1^dx_2^d (\prod_{i=1}^n x_{2+i}^{\alpha(i)})x_{3+n}^{q-1}\in C$, and thus the statement is true for $n>0$.
  
 Case 2. We assume that $n=0$:
  Now in $m_{k'}$ we set $x_1, \ldots, x_{k'-1}$ to $x_3$,
  $x_{k'}$ to $x_1$,
  and $x_{k'+1}$ to $x_2$.
 Then $x_1^{d^{k'}}x_2^{d^{k'}}x_3^{a}  \in C$ where $a = \sum_{i=1}^{k'-1} d^i$.
 Similarly as in Case 1, we have $\overline{d^{k'}} = \overline{d}$ and $\overline{a} = \overline{\sum_{i=1}^{k'-1} d^i} = \overline{\sum_{i=1}^{\phi(d'')\cdot (q-1)} d^i} =
 \overline{(q-1)\sum_{i=1}^{\phi(d'')} d^i} = q-1$.
\end{proof}

\begin{lem}\label{lem:consemcon2}
Let $q$ be a prime power such that $q-1$ is square-free.
Let $C$ be a $0$-preserving semi-affine clone with respect to $(\Z_{q-1},+,-,0)$.
Then the function 
\begin{align*}
 f \colon \{G \in \monClones{q} \mid \varphi(G) = C \} &\to \{G \cap \F_q[x_1, \ldots, x_q]\mid G\in \monClones{q}, \varphi(G) = C\}\\
 G&\mapsto f(G) := G \cap \F_q[x_1, \ldots, x_q]
 \end{align*}
is injective.
\end{lem}

\begin{proof}
We suppose that $G'\not \subseteq G$ and $\varphi(G) = \varphi(G')$.
In order to prove the statement, we show that $f(G') \not \subseteq f(G)$.
If $q=2$, we have by Proposition~\ref{pro:smallMonClon} that $G=\genMClo{\{x_1\}}$ and $G' = \genMClo{\{x_1x_2\}}$.
Then $G' \cap \F_q[x_1, x_2] \not \subseteq G \cap \F_q[x_1, x_2]$.
Let $q>2$ and let $m' \in G'\setminus G$.
By permuting variables and the closure under equivalent monomials
we assume that $m'$ is of the form $m' = (\prod_{i=1}^n x_i^{\alpha(i)}) (\prod_{j=1}^{l_1} x_{n+j}^{q-1})$ 
with $n, l_1 \geq 0$ and for all $i\leq n$, $\alpha(i) \in \{1,\ldots, q-2\}$. 
Since $\varphi(G) = \varphi(G')$,
there is a monomial
 $(\prod_{i=1}^{n+l_1} x_i^{\beta(i)}) \in G$
 where $\beta(i) \equiv_{q-1} \alpha(i) $ for each $i\leq n$
 and $\beta(i) \equiv_{q-1} 0$ for each $i\in \{n+1, \ldots, n+l_1\}$.
 Since $G$ is closed under equivalent monomials, 
 we have 
 $(\prod_{i=1}^{n} x_i^{\alpha(i)}) (\prod_{j=1}^{l_1} x_{n+j}^{\overline{\beta(n+j)}}) \in G$.
Permuting variables yields $m = (\prod_{i=1}^n x_i^{\alpha(i)})(\prod_{j=1}^{l_2} x_{n+j}^{q-1}) \in G$ for some $l_2 \in \N_0$, $l_2 < l_1$.

Case 1. We assume that $n = 0$. 
By Lemma~\ref{lem:q1Allq1} a monomial clone $D$ contains $x_1^{q-1}x_2^{q-1}$ if and only if
for all $l\in\N$, $\prod_{i=1}^l x_i^{q-1}$ lies in $D$.
Since $l_1,l_2>0$ and $m' \in G'\setminus G$, we have $x_1^{q-1}x_2^{q-1} \in G'\setminus G$,
and thus
$G' \cap \F_q[x_1, \ldots, x_q] \not \subseteq  G \cap \F_q[x_1, \ldots, x_q]$.
Hence $f(G') \not \subseteq f(G)$.

Case 2.
We assume that $n > 0$.
By Lemma~\ref{lem:q1Allq1} we have that
$(\prod_{i=1}^n x_i^{\alpha(i)})x_{n+1}^{q-1}$ lies in a monomial clone $D$ if and only if
for all $l \in \N_0$,
$(\prod_{i=1}^n x_i^{\alpha(i)})(\prod_{j=1}^l x_{n+j}^{q-1})$ lies in $D$.
Since $m\in G$ and $m' \in G'\setminus G$, we have $(\prod_{i=1}^n x_i^{\alpha(i)})x_{n+1}^{q-1} \in G'\setminus G$ and 
$m= \prod_{i=1}^n x_i^{\alpha(i)}$.
If $n \geq q$, then there occurs one of the exponents of $m=\prod_{i=1}^n x_i^{\alpha(i)}$ at least twice,
and by permuting variables and Lemma~\ref{lem:consemcon1} we get that $m\cdot x_{n+1}^{q-1}\in G$, a contradiction.
If $n<q$, we have 
 $(\prod_{i=1}^n x_i^{\alpha(i)})x_{n+1}^{q-1} \in G' \cap \F_q[x_1, \ldots, x_q] \setminus G$.
Hence $f(G') \not \subseteq f(G)$.
\end{proof}

\begin{pro}\label{pro:conSemiAff}
 Let $q$ be a prime power.
 Then $\varphi$ is a surjective map from the monomial clones on $\F_q$ to the set of $0$-preserving semi-affine clones with respect to $(\Z_{q-1},+,-,0)$.
 Furthermore the following hold:
 Let $C$ be a $0$-preserving semi-affine clone with respect to $(\Z_{q-1},+,-,0)$.
 \begin{enumerate}
  \item \label{en:conSem1} If $C = \genCloZ{\emptyset}{q-1}$, then $\{G \in \monClones{q}\mid \varphi(G) = C\} = \{\genMClo{\{x_1\}}, \genMClo{\{x_1x_2^{q-1}\}} \}$.
  \item \label{en:conSem2} If for all $G \in \monClones{q}$ with $\varphi(G) = C$ we have $x_1x_2^{q-1} \in G$,
 then we have that $\lvert \{G \in \monClones{q}\mid \varphi(G) = C\} \rvert = 1$.
 \item \label{en:conSem3} If $q-1$ is square-free, then  $\{G \in \monClones{q}\mid \varphi(G) = C\}$ is finite.
 \end{enumerate}
\end{pro}

\begin{proof}
 Let $C$ be a $0$-preserving semi-affine clone with respect to $(\Z_{q-1},+,-,0)$.
 Now we define
 \begin{align*}
 G' := \{&\prod_{i=1}^n x_i^{r(i)}\mid ((y_1, \ldots, y_n) \mapsto \sum_{i=1}^n [r(i)] y_i)\in C,\exists i\leq n\colon r(i) \neq  0 \}.
 \end{align*}
 The condition ($\exists i\leq n\colon r(i) \neq  0 $) is required
 to exclude $x_1^0 = 1$ as a monomial in $G'$.
 
 First, we show that $\varphi$ is surjective.
 We have that $C$ is a $0$-preserving semi-affine clone with respect to $(\Z_{q-1},+,-,0)$.
 Therefore we see that $\genMClo{\{x_1\}} \subseteq G'$,
 $G'$ is closed under substitution of monomials,
 and $G'$ is closed under equivalent monomials.
 Hence, $G'$ is a monomial clone on $\F_q$.
 Furthermore, we have $\varphi(G') = C$ and $G'$ is the largest $G \in \monClones{q}$ such that $\varphi(G) = C$.

 Now we show the second part of the statement
 and start with Item \eqref{en:conSem1}.
  If $C = \genCloZ{\emptyset}{q-1}$, then $C$ is the clone of projections.
  Then we see by Lemma~\ref{lem:q1Allq1} that 
  $\{G \in \monClones{q}\mid \varphi(G) = C\} = \{\genMClo{\{x_1\}}, \genMClo{\{x_1x_2^{q-1}\}} \}$.
 For Item \eqref{en:conSem2} we observe the following:
 Let $G \in \monClones{q}$ with $\varphi(G) = C$.
 Now let $m = \prod_{i=1}^n x_i^{\alpha(i)}\in G$.
 Since $m' = x_1x_2^{q-1}\in G$
 we get $m'(m, x_{n+1}) = (\prod_{i=1}^n x_i^{\alpha(i)}) x_{n+1}^{q-1}\in G$.
 Now we see by Lemma~\ref{lem:q1Allq1} and by identifying variables
 that $G = G'$.
 Now we show Item~\eqref{en:conSem3}.
 There are only finitely many different non-equivalent monomials in $\F_q[x_1, \ldots, x_q]$.
 Therefore, we get that the co-domain of the function $f$ of Lemma~\ref{lem:consemcon2} is finite.
 By Lemma~\ref{lem:consemcon2} we know that $f$ is injective, and thus
 the domain
 $\{G \in \monClones{q}\mid \varphi(G) = C\}$ of $f$ is finite.
\end{proof}

Let $q$ be a prime power.
Concluding this section, we note
 that the investigation of the lattice of monomial clones on $\F_q$ is equivalent
to the problem of finding the subclones of the clone of the algebra
$(\F_q, \cdot)$,
which is isomorphic to $(\Z_{q-1}\cup \{-\infty\}, +)$, where we extend the operation $+$ on $\Z_{q-1}$ by
$(-\infty) + b = b + (-\infty) = -\infty$
for all $b\in \Z_{q-1}\cup \{-\infty\}$.

\section{The top and the bottom of the lattice of monomial clones}\label{sec:TopBot}

Let $q>2$ be a prime power.
From \cite{MP:TCCSM} we know that the lattice of monomial clones where the clones are generated by one single
binary monomial $x_1x_2^b$ is isomorphic to the divisor lattice of $q-1$.
We already see in Figure \ref{fig:pM1} for $q=3$ that the set $\{ \genMClo{\{x_1x_2^b\}} \mid b \textmd{ divides } q-1 \}$
is not equal to the interval $[\genMClo{\{x_1x_2^2\}}, \genMClo{\{x_1x_2\}}]$ in the lattice of monomial clones on $\F_q$.
In this section we consider the top and the bottom of the lattice of monomial clones on $\F_q$ for any prime power $q>2$.
We start to describe the atoms of the lattice and then we continue with the coatoms.

\subsection{Atoms}

Investigations on the atoms of the lattice of monomial clones generated by a unary function can be found in \cite{MP:TCCSM}.
We will use one of their results in Corollary~\ref{cor:atoms}.

\begin{lem}\label{lem:smIdClone}
  Let $q$ be a prime power
  and let $C \neq \genMClo{\{x_1\}} $ be an idempotent monomial clone on $\F_q$.
  Then $C$ contains $x_1x_2^{q-1}$.
\end{lem}

\begin{proof}
 Since $C \neq \genMClo{\{x_1\}}$, we have that
 there is an $n\in\N$ with $n\geq 2$, and
 there are $\alpha(1), \ldots, \alpha(n) \in \N$
 such that
 $m=\prod_{i=1}^n x_i^{\alpha(i)} \in C$.
 Now we identify all variables $x_i$ where $i\geq 3$ with $x_2$
 and get $x_1^{\alpha(1)} x_2^{\sum_{i=2}^n\alpha(i)} \in C$.
 Since $m$ is idempotent we have $\alpha(1) + \sum_{i=2}^n \alpha(i) \equiv_{q-1} 1$,
 and thus $C$ contains the monomial $x_1^{\overline{\alpha(1)}}x_2^{q-\overline{\alpha(1)}}$.
 By \cite[Theorem 2.4]{GKS:MC} we now get $x_1x_2^{q-1} \in \genMClo{\{x_1^{\overline{\alpha(1)}}x_2^{q-\overline{\alpha(1)}}\}}$.
\end{proof}

\begin{cor}\label{cor:atoms}
  Let $q$ be a prime power.
 Then the atoms of the lattice of monomial clones on $\F_q$ are given by $\genMClo{\{x_1x_2^{q-1}\}}$ and
 $\genMClo{\{x_1^{s}\}}$
 where $2 \leq s\leq q-1$ and $s^P \equiv_{q-1} 1$ for some prime $P$ or $s\cdot s \equiv_{q-1} s$.
\end{cor}

\begin{proof}
Let $C$ be an atom of the lattice of monomial clones on $\F_q$.
If $C$ contains a non-idempotent monomial, then, by variable identification and reduction
of the exponent, $C$ contains a monomial $x_1^t$ for some $t\in\N$ with $2\leq t \leq q-1$.
As $C$ is an atom and $x_1^t \not\in \genMClo{\{x_1\}}$, we have $C= \genMClo{\{x_1^t\}}$.
The result now follows from \cite[Lemma 3.1]{MP:TCCSM}, where more details can be found in \cite{MP:MCLOCRES}.
Now we assume that $C$ contains only idempotent monomials.
Since $\genMClo{\{x_1\}}$ is a proper subset of $C$, 
we have $x_1x_2^{q-1} \in C$ by Lemma~\ref{lem:smIdClone}, and thus $C = \genMClo{\{x_1x_2^{q-1}\}}$.
\end{proof}

\subsection{Coatoms}

\begin{lem}\label{lem:xkxlgcd}
 Let $q$ be a prime power and let $C$ be a monomial clone on $\F_q$.
 Let $k, l \in \N$.
 Then $x_1 \cdots x_{1+k} \in C$ and $x_1 \cdots x_{1+l} \in C$
 if and only if $x_1 \cdots x_{1+\gcd(k,l)} \in C$.
\end{lem}

\begin{proof}
By Lemma~\ref{lem:combRule}
we have $x_1 \cdots x_{1+t\cdot\gcd(k,l)} \in C$ for all $t\in \N$,
and thus the ``if''-direction holds.
Now we prove the ``only if''-direction.
Let $k' := \gcd(k, q-1)$ and $l' := \gcd(l, q-1)$.
By Lemma~\ref{lem:x1xkThenAllK} 
we have for all $t_1, t_2 \in \N_0$,
that $m^1_{t_1} := x_1 \cdots x_{1+t_1\cdot k'}$
and $m^2_{t_2}:=x_1 \cdots x_{1+t_2 \cdot l'}$ lie in $C$.
Let $t_1, t_2 \in \N_0$ with $t_1>0$ or $t_2>0$.
Then
\begin{align*}
 m^1_{t_1}&(x_1, \ldots, x_{t_1\cdot k'}, m^2_{t_2}(x_{1+t_1\cdot k'},\ldots, x_{1+t_1\cdot k' +t_2\cdot l'})) 
 = x_1\cdots x_{1+t_1\cdot k' + t_2\cdot l'} \in C.
\end{align*}

We proceed similar as in Lemma~\ref{lem:x1xkThenAllK}.
By identifying variables we get that the monomial
$x_1 x_2^{t_1\cdot k' + t_2\cdot l'}$ lies in $C$.
There are $u_1, u_2 \in \Z$ such that $u_1\cdot k' + u_2 \cdot l' = \gcd(k,l,q-1)$,
since $k' = \gcd(k, q-1)$ and $l'= \gcd(l, q-1)$.
Hence there are  $t_1', t_2' \in \N$
such that $\overline{t_1'\cdot k'  + t_2' \cdot l' } = \gcd(k, l, q-1)$.
This means that $x_1x_2^{\gcd(k,l,q-1)} \in C$,
and by identifying variables we get $x_1^{1+\gcd(k,l,q-1)} \in C$.
We have $x_1\cdots x_{1+k}\in C$ and thus we get by Lemma~\ref{lem:twoInvThenx1xq} that $x_1\cdots x_q \in C$.
Since $x_1\cdots x_q \in C$ and $x_1^{1+\gcd(k,l,q-1)} \in C$
we get by Corollary~\ref{cor:x1xqxathenx1xa} that $x_1\cdots x_{1+\gcd(k,l, q-1)} \in C$.
By Lemma~\ref{lem:combRule} we get now that
$ x_1\cdots x_{1+ t\cdot \gcd(k,l,q-1)} \in C$ for all $t\in \N$,
and thus we have  $x_1 \cdots x_{1 +\gcd(k,l)}\in C$.
\end{proof}

Theorem~\ref{thm:topIsoDivLat} shows that the lattice between $\genMClo{\{x_1\cdots x_q\}}$ and $\genMClo{\{x_1x_2\}}$ is antiisomorphic to
the divisor lattice of $q-1$.
We denote $([\genMClo{\{x_1 \cdots x_q\}}, \genMClo{\{x_1x_2\}}], \vee, \cap)$
 by $\mathcal{X}(q)$ and the lattice of (positive) divisors of $q-1$
 by $\mathcal{D}(q-1)$.

\begin{thm}\label{thm:topIsoDivLat}
 Let $q$ be a prime power,
 and let $\mathcal{C} \colon \mathcal{D}(q-1) \to \mathcal{X}(q)$, $\mathcal{C}(a):= \genMClo{\{x_1 \cdots x_{a+1} \}}$.
 Let $a, b\in \mathcal{D}(q-1)$.
 Then $b\mid a$ if and only if $\mathcal{C}(a) \subseteq \mathcal{C}(b)$.
 Furthermore, $\mathcal{C}$ is surjective, and thus $\mathcal{C}$ is a lattice isomorphism from $\mathcal{D}(q-1)$ to $\mathcal{X}(q)$.
\end{thm}

\begin{proof}
The claim is obvious for $q=2$, so we assume $q\geq 3$.
 If $b \mid a$, there exists $t \leq q-1$ such that $t\cdot b = a$.
 Then we get by Lemma~\ref{lem:combRule}, that $x_1 \cdots x_{1+a} = x_1 \cdots x_{1+t\cdot b} \in \mathcal{C}(b)$,
 and thus $\mathcal{C}(a) \subseteq \mathcal{C}(b)$.
 Now we assume $\mathcal{C}(a) \subseteq \mathcal{C}(b)$.
 By Lemma~\ref{lem:x1xkThenAllK} we have that $\genMClo{\{x_1\cdots x_{1+b}\}}$
 precisely contains all monomials where the sum of exponents is congruent to $1$ modulo $b$.
 Therefore, we have that
 $b$ is the smallest natural number $k\in\N$ such that $x_1\cdots x_{1+k} \in \mathcal{C}(b)$.
 By Lemma~\ref{lem:xkxlgcd} we get
 $x_1\cdots x_{1+\gcd(a,b)} \in \mathcal{C}(b)$.
 Since $b$ is the smallest element $k\in\N$ such that
 $x_1 \cdots x_{1+k} \in \mathcal{C}(b)$,
 we have $\gcd(a,b) = b$, and thus $b\mid a$.
 Hence $\mathcal{C}$ is an order embedding.
 
 To prove surjectivity, we show the following:
 If $C$ is a monomial clone on $\F_q$ that contains $x_1\cdots x_q$,
 then there exists
 an $a\in \N$ with $a\mid q-1$ such that $C= \mathcal{C}(a)$.
 First, we show that for all monomials $m$
 there exists an $a\in \N$ with $a\mid q-1$ such that $\genMClo{\{m, x_1\cdots x_q\}} = \mathcal{C}(a)$.
 Let $m= x_1^{\alpha(1)} \cdots x_n^{\alpha(n)}$.
 Then $x_1^{\alpha} \in \genMClo{\{m, x_1\cdots x_q\}}$ where
 $\alpha = \overline{\sum_{i=1}^n \alpha(i)}$.
 If $\alpha = q-1$, then $x_1^{q-1}$ and $x_1\cdots x_q\in C$,
 so we have $x_1 x_2x_3^{q-1}\cdots x_q^{q-1}\in C$, and thus
 Lemma~\ref{lem:cutQ1} yields $x_1 x_2 \in \genMClo{\{m, x_1\cdots x_q\}}$.
 Hence, $\genMClo{\{m, x_1\cdots x_q\}}=\genMClo{\{x_1x_2\}}$.
 If $\alpha = 1$, then $\genMClo{\{m, x_1\cdots x_q\}}= \genMClo{\{x_1\cdots x_q\}}= \mathcal{C}(q-1)$ by Lemma~\ref{lem:maxIdClone}.
 Now we assume that $1<\alpha<q-1$.
 By Corollary~\ref{cor:x1xqxathenx1xa} we get
 $x_1 \cdots x_{\alpha} \in \genMClo{\{m, x_1\cdots x_q\}}$ and thus $\genMClo{\{x_1 \cdots x_{\alpha}, x_1\cdots x_q\}} \subseteq \genMClo{\{m, x_1\cdots x_q\}}$.
 By Lemma~\ref{lem:xkxlgcd},
 $x_1 \cdots x_{1 + \gcd(\alpha-1, q-1)} \in \genMClo{\{x_1 \cdots x_{\alpha}, x_1\cdots x_q\}} \subseteq \genMClo{\{m, x_1\cdots x_q\}} $.
 On the other hand, it follows from Lemma~\ref{lem:x1xkThenAllK} that $m$ and $x_1 \cdots x_q$
 lie in  $\genMClo{\{x_1 \cdots x_{1 + \gcd(\alpha-1, q-1)}\}}$.
 Hence we have $\genMClo{\{m, x_1\cdots x_q\}} = \mathcal{C}(\gcd(\alpha-1, q-1))$.
 We showed that for all monomials $m$
 there exists an $a\in \N$ with $a\mid q-1$ such that $\genMClo{\{m, x_1\cdots x_q\}} = \mathcal{C}(a)$.
 
 Let $C$ be a monomial clone on $\F_q$ that contains $x_1\cdots x_q$.
 We know that for every $m\in C$ there is some $a_m\in \mathcal{D}(q-1)$
 such that $\mathcal{C}(a_m) = \genMClo{\{m, x_1 \cdots x_q\}} \subseteq C$.
 The set $A= \{a_m \mid m\in C\} \subseteq \mathcal{D}(q-1)$ is finite, so $a:= \gcd(A)$ can be computed. 
From Lemma~\ref{lem:xkxlgcd} it follows that $\mathcal{C}(a)\subseteq C$,
and conversely for every $m \in C$ we have $m \in \mathcal{C}(a_m) \subseteq \mathcal{C}(a)$.
Thus $C = \mathcal{C}(a)$.
\end{proof}

\begin{lem}\label{lem:dup2}
 Let $q$ be a prime power and let $C$ be a monomial clone on $\F_q$.
 Let $n\in\N$ with $n\geq 2$, and let $\alpha(1),\ldots, \alpha(n)\in \N$ and let $\gamma \in \N_0$.
  We assume that $m = x_1^{\alpha(1)} \cdots x_n^{\alpha(n)}x_{n+1}^{\gamma} \in C$.
   Then for all $t\in \N_0$ there is a $k\colon \{1, \ldots, t\} \to \N$
   and there is a $\gamma' \in \N_0$
  such that
  $$x_1^{\alpha(1)} x_2^{\alpha(2)^{t+1}} x_3^{\alpha(3)} \cdots x_n^{\alpha(n)}x_{n+1}^{\alpha(1)\alpha(2)^{k(1)}}x_{n+2}^{\alpha(1)\alpha(2)^{k(2)}}\cdots x_{n+t}^{\alpha(1)\alpha(2)^{k(t)}} x_{n+t+1}^{\gamma'} \in C.$$
\end{lem}

\begin{proof}
 We proceed by induction on $t\in\N_0$.
 If $t=0$, then $m \in C$ and the statement is true.
 Let $t>0$.
 We assume that the induction hypothesis holds for $t-1$.
 By the induction hypothesis
 we have 
 $$m'(x_1, \ldots, x_{n+t}) = x_1^{\alpha(1)}x_2^{\alpha(2)^{t}}x_{3}^{\alpha(3)}\cdots x_{n}^{\alpha(n)}x_{n+1}^{\beta(1)} \cdots x_{n+t-1}^{\beta(t-1)}x_{n+t}^{\gamma'} \in C,$$
 where for all $i\leq t-1$ there is a $k(i) \in \N$,
 such that $\beta(i) = \alpha(1) \alpha(2)^{k(i)}$,
 and $\gamma'\in \N_0$.
 By substitution of monomials we have that
  \begin{equation*}
  \begin{aligned}
 m(&x_1,m'(x_{n+1}, x_2, \underbrace{x_{n+t+1}, \ldots, x_{n+t+1}}_{n-2},x_{n+2}, x_{n+3}, \ldots, x_{n+t+1}), x_{3} \ldots, x_n, x_{n+t+1})\\
 &= x_1^{\alpha(1)}x_2^{\alpha(2)^{t+1}}x_3^{\alpha(3)}\cdots x_n^{\alpha(n)}x_{n+1}^{\alpha(1)\alpha(2)}x_{n+2}^{\alpha(2)\beta(1)}\cdots x_{n+t}^{\alpha(2)\beta(t-1)} x_{n+t+1}^{\gamma''} \in C,
  \end{aligned}
 \end{equation*}
 for some $\gamma'' \in \N_0$.
 This finishes the induction step.
\end{proof}

\begin{lem}\label{lem:gcdprod}
 Let $n\in\N$ with $n>1$, 
 and let $\alpha_1, \ldots, \alpha_n, \beta \in\N$ be such that
 for all $E \subseteq \{1,\ldots, n\}$ with $\lvert E \rvert = n-1$,
 we have
$\gcd(\{\alpha_i\mid i \in E \} \cup \{\beta\}) = 1$.
Let $k,l \in  \N^{\{1,\ldots, n\}\times \{1\ldots, n\}}$.
Then $\gcd(\{\alpha_i^{k(i,j)}\alpha_j^{l(i,j)} \mid 1 \leq i,j \leq n, i < j\}\cup \{ \beta\} ) = 1$.
\end{lem}

\begin{proof}
By assumption, no prime divisor of $\beta$ appears in $n-1$ of the $\alpha(i)$'s.
Therefore, every prime factor of $\beta$ appears in at most $n-2$ of the $\alpha(i)$'s, and thus 
does not appear in the other $2$.
Hence $\gcd(\{\alpha_i^{k(i,j)}\alpha_j^{l(i,j)} \mid i,j \leq n, i < j\}\cup \{ \beta\} ) = 1$.
\end{proof}

\begin{lem}\label{lem:3gcdx1xq}
 Let $q$ be a prime power and let $C$ be a monomial clone on $\F_q$.
 Let $n\in\N$ with $n>1$ and let $\alpha(1),\ldots,\alpha(n)\in \N$ be
 such that for all $E \subseteq \{1,\ldots, n\}$ with $\lvert E \rvert = n-1$
we have
$\gcd(\{\alpha(i)  \mid i \in E\}\cup\{q-1\}) = 1$.
Let $\gamma \in \N_0$.
 If $x_1^{\alpha(1)}\cdots x_{n}^{\alpha(n)}x_{n+1}^{\gamma} \in C$,
 then $x_1\cdots x_q \in C$.
\end{lem}

\begin{proof}
Let $m := \prod_{i=1}^{n'} x_i^{\beta(i)}$ be a monomial.
We define for all $i, j\in \{1,\ldots, n\}$ with $i < j$,
$$E(m,i,j):= \{l \in\N \mid  n < l\leq n', \exists e_1, e_2 \in \N \colon \beta(l) = \alpha(i)^{e_1}\alpha(j)^{e_2}\}.$$

Let $t\in \N$.
We proceed by induction on $k\in\N\setminus\{1\}$
to show that there exist $e \in \N^n$, $n' \in \N$
and $h = \prod_{i=1}^{n'} x_i^{\beta(i)} \in C$
with $\beta(1), \ldots, \beta(n') \in \N$ 
such that 
for all $i \in \{1,\ldots, n\}$
we have $\beta(i) = \alpha(i)^{e(i)}$
and 
for all 
$i, j\in \{1,\ldots, n\}$ with $i< j$ 
and $(i-1)\cdot n + j \leq k$
we have
$\lvert E(h,i,j) \rvert \geq t$.

We assume that $k = 2$. Let $i',j'\in \{1,\ldots, n\}$ with $i' < j'$ be
such that $(i'-1)\cdot n + j' \leq k$.
Then $i' = 1$ and $j'=2$, since $n>1$ and $i',j' > 0$.
Since $C$ contains
$x_1^{\alpha(1)}\cdots x_{n}^{\alpha(n)}x_{n+1}^{\gamma}$,
 we get the desired $h$ by Lemma~\ref{lem:dup2}.
 
 Now let $k>2$.
 By the induction hypothesis
 there exist $e' \in \N^n$, $n'\in\N$
 and $h' = \prod_{i=1}^{n'} x_i^{\beta(i)} \in C$
with $\beta(1), \ldots, \beta(n') \in \N$ 
such that
for all $i \in \{1,\ldots, n\}$
we have $\beta(i) = \alpha(i)^{e'(i)}$
and
for all 
$i, j\in \{1,\ldots, n\}$ with $i< j$ and $(i-1)\cdot n + j \leq k-1$
we have $\lvert E(h',i,j) \rvert \geq t$.
 If there are  $i', j'\in \{1,\ldots, n\}$ such that $(i'-1)\cdot n + j' = k$, they are uniquely determined.
If $i' \geq  j'$ we set $h:= h'$.
Now we assume that $i' < j'$.
We permute in $h'$ the variable $x_1$ with $x_{i'}$
and the variable $x_2$ with $x_{j'}$ to get $h''$.
Then we get the desired $h$ by applying Lemma \ref{lem:dup2} to $h''$ and by
permuting back the variable $x_{i'}$ with $x_1$
and the variable $x_{j'}$ with $x_2$. This concludes the induction step.

By setting $t:= 2\cdot q^2$ and by choosing $k=n^2$ 
there is an $h = x_1^{\beta(1)}\cdots x_{n'}^{\beta(n')}\in C$
such that
for all $i, j\in \{1,\ldots, n\}$ with $i < j$,
we have $\lvert E(h, i, j) \rvert \geq t = 2\cdot q^2$.
For every $i,j \leq n$ with $i<j$, we define $A(i,j):=\{\overline{\alpha(i)^{e_1}\alpha(j)^{e_2}}\mid e_1, e_2 \in \N\}$.
We have $\lvert A(i,j) \rvert \leq q$ for each $i,j \leq n$ with $i<j$.
Let $i,j\leq n$ with $i<j$.
If for all $a \in A(i,j)$, we have
$\lvert \{l \in \N \mid n<l \leq n', a = \overline{\beta(l)} \} \rvert \leq 2\cdot q - 1$,
then  
$\lvert E(h, i, j) \rvert 
\leq \sum_{a\in A(i,j)} \lvert \{l \in \N \mid n<l \leq n', a = \overline{\beta(l)} \}\rvert 
\leq q\cdot (2\cdot q - 1) = 2\cdot q^2 - q < 2 \cdot q^2$
which is a contradiction, since $\lvert E(h, i, j) \rvert \geq 2\cdot q^2$.
Therefore, 
for all $i,j\leq n$ with $i<j$
there is an $a(i,j) \in A(i,j)$
such that $\lvert \{l \in \N \mid n<l \leq n', a(i,j) = \overline{\beta(l)} \} \rvert \geq 2 \cdot q$.
We define $g := \gcd(\{a(i,j) \mid (i,j)\in\{1, \ldots, n\}^2, i <  j\} \cup \{q-1\})$.
Let $I \subseteq \{1,\ldots, n\}^2$
be such that 
$ \{a(i,j) \mid (i,j)\in\{1, \ldots, n\}^2, i <  j\} =
  \{a(i, j) \mid (i,j)\in I\}$
  and for all $(i,j),(i',j') \in I$ with $(i,j) \neq (i',j')$
  we have $a(i,j) \neq a(i',j')$.
  Since for all $E \subseteq \{1,\ldots, n\}$ with $\lvert E \rvert = n-1$
we have
$\gcd(\{\alpha(i)  \mid i \in E\}\cup\{q-1\}) = 1$,
we get by Lemma~\ref{lem:gcdprod}
that $g = 1$.
Let $\tilde{t} \colon I\to \{1, \ldots, q-1\}$ be such that
$$ \overline{\sum_{(i,j)\in I} \tilde{t}(i,j) a(i,j)} = g = 1.$$
Since for all $i,j\leq n$ with $i<j$ we have
$\lvert \{k \leq n' \mid  \overline{\beta(k)} = a(i,j) \} \rvert \geq 2 \cdot q$,
there are two functions $X$ and $X'$ from $I$ 
to the finite subsets of $\{x_l \mid l\in \N\}$
such that
for all $(i,j) \in I$
we have
$\{a(i,j)\} = \{\overline{\beta(l)} \mid l \in \N, x_l \in X(i,j)\} =\{\overline{\beta(l)} \mid l \in \N, x_l \in X'(i,j)\}$,
$\lvert X(i,j)\rvert =  \lvert X'(i,j)\rvert = \tilde{t}(i,j)$
and $X(i,j) \cap X'(i,j) = \emptyset$.
We mention that for $(i_1,j_1) \neq (i_2, j_2)$,
the sets $X(i_1,j_1)$, $X'(i_1, j_1)$, $X(i_2,j_2)$ and $X'(i_2, j_2)$
are pairwise disjoint.
For all $(i, j) \in I$
we set the variables of $X(i,j)$ to $x_1$,
and 
we set the variables of $X'(i,j)$ to $x_2$.
All other variables of $h$ with exponent not equal to zero are set to $x_3$.
Now we have
$x_1^{S} x_2^{T}x_3^{\gamma}\in C$,
where $\gamma \in \N_0$ and 
$$S := \sum_{(i,j)\in I} \left(\sum_{x_l \in X(i,j)} \beta(l)\right), \textmd{ and } T := \sum_{(i,j) \in I} \left(\sum_{x_l \in X'(i,j)} \beta(l)\right).$$
For all $(i,j) \in I$ we have
$$\sum_{x_l \in X(i,j)} \overline{\beta(l)} = \sum_{x_l \in X(i,j)} a(i, j) = \tilde{t}(i,j) \cdot a(i,j)$$
and
$$\sum_{x_l \in X'(i,j)} \overline{\beta(l)} = \sum_{x_l \in X'(i,j)} a(i, j) = \tilde{t}(i,j) \cdot a(i,j),$$
hence $\overline{S} = \overline{T} = \overline{\sum_{(i,j)\in I} \tilde{t}(i,j)a(i,j)} = g = 1$.
Since $C$ is closed under equivalent monomials, we have
$x_1^{g}x_2^{g}x_3^{\gamma} = x_1 x_2x_3^{\gamma} \in C$.
Now Lemma~\ref{lem:2var1x1xq} yields that $x_1\cdots x_q$ lies in $C$.
\end{proof}

Let $q>2$ be a prime power, let $l\in\N$ and let $P_1<\ldots < P_l$ be the primes that divide $q-1$.
Let $D \subseteq \{1,\ldots,l\}$ be nonempty.
Now we define for $D$ the monomial clone $\addCo{D}$ by
$\genMClo{\{x_1^t, x_1^{P_i} \cdots x_{n}^{P_i},
x_1x_2^{T}
\mid n,t\in\N, i\in D\}}$,
where $T$ is an abbreviation of $\prod_{t\in D}P_t$.

\begin{lem}\label{lem:addCo}
  Let $q>2$ be a prime power, let $l\in\N$ and let $P_1<\ldots < P_l$ be the primes that divide $q-1$.
  Let $D \subseteq \{1,\ldots, l\}$ be nonempty. Let $n\in\N$.
  Then for all $\alpha(1), \ldots, \alpha(n) \in \N$ we have that $x_1^{\alpha(1)}\cdots x_n^{\alpha(n)} \in \addCo{D}$ 
  if and only if
  there exists $i\in D$ such that for all $j\leq n$ we have that $P_i$ divides $\alpha(j)$,
  or
  $\lvert \{i \in \N \mid i\leq n, \prod_{t\in D} P_t \textmd{ divides } \alpha(i)\}  \rvert \geq n-1$.
\end{lem}

\begin{proof}
Let $T:= \prod_{t\in D} P_t$.
Now we define
\begin{align*}
C := \{\prod_{i = 1}^n x_i^{\alpha(i)} \mid &n\in \N, \exists j\leq n\colon \alpha(j) \neq 0, 
(\exists i \in D \forall j \leq n\colon  \alpha(j) \equiv_{P_i} 0) \vee\\&
(\exists j\leq n \forall i\in \{1, \ldots, n\}\setminus \{j\}\colon \alpha(i)\equiv_{T} 0 )\}.
\end{align*}
We show that 
\begin{equation}\label{cl:addCo1}
\tag{Claim 1}
C \text{ is a monomial clone on } \F_q.
\end{equation}
We easily see that $\genMClo{\{x_1\}} \subseteq C$,
and $C$ is closed under equivalent monomials, since for all $i\leq l$, $P_i$ divides $q-1$.
Now we show that $C$ is closed under substitution of monomials.
Let $m= \prod_{i=1}^n x_i^{\alpha(i)} \in C$ and let $m_1, \ldots, m_n \in C$.
Now we show that $m' := m (m_1, \ldots, m_n)$ lies in $C$.
Now let $n'\in\N$ and let $\beta(1), \ldots, \beta(n') \in \N_0$ be
such that $\prod_{i=1}^{n'}x_i^{\beta(i)} = m'$.

Case 1:
There exists an $i\in D$ such that for all $j\leq n$ we have
$\alpha(j) \equiv_{P_i} 0$:
Then we see that for all $j\leq n'$ we have
$\beta(j) \equiv_{P_i} 0$, and thus $m' \in C$.

Case 2: 
There exists $j\leq n$ such that for all $i\in \{1, \ldots, n\}\setminus\{j\}$
we have $\alpha(i)\equiv_{T} 0$:
Let $n_j \in\N$ and let $\gamma(1), \ldots, \gamma(n_j) \in \N$ be
such that $\prod_{i=1}^{n_j} x_i^{\gamma(i)} = m_j$.

Case 2a: 
There exists an $i\in D$ such that for all $j\leq n_j$ we have
$\gamma(j) \equiv_{P_i} 0$:
We see that for all $j\leq n'$ we have
$\beta(j) \equiv_{P_i} 0$, since $P_i$ divides $T$, and thus $m' \in C$.

Case 2b: 
There exists $j'\leq n_j$ such that for all $i\in \{1, \ldots, n_j\}\setminus\{j'\}$
we have $\gamma(i)\equiv_{T} 0$:
We see that for all $i\in \{1, \ldots, n'\}\setminus\{j'\}$ we have
$\beta(i)\equiv_{T} 0$, and thus $m' \in C$.

Hence $C$ is closed under substitution of monomials, and thus \ref{cl:addCo1} holds.

The next goal is to show that 
\begin{equation}\label{cl:addCo2}
\tag{Claim 2}
 \addCo{D} = C.
\end{equation}

``$\subseteq$'' is clear since the generators of $\addCo{D}$ lie in $C$.
For ``$\supseteq$'' we observe the following.
Let $m = \prod_{i=1}^n x_i^{\alpha(i)} \in C$.
We show that $m\in \addCo{D}$.
This is certainly true for $n=1$, so assume $n\geq 2$.
Let $I:= \{i \in \N \mid i\leq n, \alpha(i)=0\}$.
We have $J:=\{1, \ldots, n\}\setminus I \neq \emptyset$.
Let $m' := \prod_{i \in J } x_i^{\alpha(i)}\prod_{i\in I}x_i^{q-1}$.
Now let $\beta(1), \ldots, \beta(n) \in \N$ be
such that for all $i\in J$ we have $\beta(i) = \alpha(i)$
and for all $i\in I$ we have $\beta(i) =q-1$.
Then $m'= \prod_{i=1}^n x_i^{\beta(i)}$.
Now we show that $m' \in \addCo{D}$.

Case 1.
There exists an $i\in D$ such that for all $j\leq n$ we have
$\alpha(j) \equiv_{P_i} 0$:
Since $q-1 \equiv_{P_i} 0$ we have
for all $j\leq n$ that
$\beta(j) \equiv_{P_i} 0$:
We easily see that $m' \in \genMClo{\{x_1^t, x_1^{P_i} \cdots x_n^{P_i} \mid t\in \N\}} \subseteq \addCo{D}$.

Case 2. 
There exists $j\leq n$ such that for all $i\in \{1, \ldots, n\}\setminus\{j\}$
we have $\alpha(i)\equiv_{T} 0$:
Since $q-1 \equiv_{T} 0$ we have
for all $i\in \{1, \ldots, n\}\setminus\{j\}$ that
$\beta(i)\equiv_{T} 0$.
We see by Lemma~\ref{lem:combRule} and by permuting variables that
$m'\in \genMClo{\{x_1^t, x_1 x_2^T \mid t\in \N\}}\subseteq \addCo{D}$.

Since $m' \in \addCo{D}$, we get by Lemma \ref{lem:cutQ1} that $m \in \addCo{D}$, and thus 
\ref{cl:addCo2} holds.

\end{proof}

\begin{lem}\label{lem:noOthCo}
Let $q>2$ be a prime power, let $l\in\N$ and let $P_1<\ldots < P_l$ be the primes that divide $q-1$.
Let $C$ be a monomial clone on $\F_q$.
We assume that for all nonempty $D \subseteq \{1, \ldots, l\}$,
$C \not \subseteq \addCo{D}$.
Then $C$ contains the monomial $x_1 \cdots x_q$.
\end{lem}

\begin{proof}
 If for all $m\in C$ of width greater than $1$ there exists an $i\leq l$ such that $P_i$ divides all exponents of $m$,
then we have by Lemma \ref{lem:addCo} that $C \subseteq \addCo{\{1,\ldots, l\}}$, which is a contradiction to our assumption that
$C \not \subseteq \addCo{\{1,\ldots, l\}}$.
Hence, the set $C'$ defined by
\begin{align*}
 C' := \{x_1^{\alpha(1)}\cdots x_n^{\alpha(n)}\in C \mid &n>1, \alpha(1), \ldots, \alpha(n)\in \N, \\ &\forall i\leq l \; \exists j\leq n\colon P_i \textmd{ does not divide } \alpha(j)\}
\end{align*}
is nonempty.

Let $m = x_1^{\alpha(1)}\cdots x_n^{\alpha(n)}$ be a monomial with $\alpha(1), \ldots, \alpha(n)\in \N$.
Then we define for $i\in \{1, \ldots, n\}$,
$$d'(m, i) := \gcd(\{\alpha(j) \mid j \in \{1, \ldots, n\}\setminus\{i\} \} \cup \{P_1 \cdots P_l\}).$$
Now let 
$$d(m) := \max(\{d'(m,i) \mid i \in \{1, \ldots, n\}\}).$$
We have that $d(m)$ is equal to $1$ or equal to a product of distinct primes that divides $q-1$.
Now we assume that 
$m = x_1^{\alpha(1)}\cdots x_n^{\alpha(n)} \in C'$ has the property that 
for all $m' \in C'$ we have
\begin{equation}\label{prop:coatom1}
  \tag{P1}
 d(m) \leq d(m').
\end{equation}
This means that $d(m)$ is minimal with respect to $\leq$.

Case 1. We assume that $d(m)=1$:
Since $d(m) = \max(\{d'(m,i) \mid i \in \{1, \ldots, n\}\})$ and $P_1, \ldots, P_l$ are all primes that divide $q-1$,
we have
 $\gcd(\{\alpha(i)\mid i \in E\} \cup \{ q-1\} ) = 1$ for all $E \subseteq \{1,\ldots, n\}$ with $\lvert E \rvert = n-1$.
Since $n>1$, Lemma~\ref{lem:3gcdx1xq} now yields $x_1\cdots x_q \in C$.

Case 2. Now we assume $1 < d(m) \leq P_1\cdots P_l$:
We show that this assumption leads to a contradiction to the minimality of $d(m)$. 
Let $k \leq n$ be
such that
$d(m)=\gcd(\alpha(1), \ldots, \alpha(k-1), \alpha(k+1),\ldots, \alpha(n), P_1\cdots P_l)$.
Let $D \subseteq \{1,\ldots, l\}$ be such that $d(m) = \prod_{i\in D} P_i$.
Since $d(m) > 1$, we have that $D \neq \emptyset$, 
and thus we have by Lemma~\ref{lem:addCo} that $m \in \addCo{D}$.
Since $m\in C'$, we have for all $i\in D$ that $P_i$ does not divide $\alpha(k)$,
which means $\gcd(\alpha(k), P_i) = 1$ for all $i \in D$,
and thus $\gcd(\alpha(k), d(m)) = 1$.
We define $D' := \{1, \ldots, l\}\setminus D$.
For all $i'\in D'$
there exists a $j\leq n$ with $j\neq k$
such that $\gcd(\alpha(j), P_{i'}) = 1$,
since $d(m)=\gcd(\alpha(1), \ldots, \alpha(k-1), \alpha(k+1),\ldots, \alpha(n), P_1\cdots P_l) = \prod_{i\in D} P_i$.

Case 2a. We assume that $C' \subseteq \addCo{D}$:
Since $C \not \subseteq \addCo{D}$,
there is some $m_1 = x_1^{\beta(1)} \cdots x_{n_1}^{\beta(n_1)}\in C$
with $\beta(1), \ldots, \beta(n_1)\in \N$
such that $m_1  \not \in \addCo{D}$.
 By substitution of monomials we get 
 \begin{align*}
 m_2(x_1, \ldots, x_{n +n_1 -1}) &:= m(x_1 , \ldots x_{k-1}, m_1(x_k, \ldots, x_{k+n_1-1}) ,x_{k+n_1}, \ldots, x_{n +n_1 -1}) \\
  &= \prod_{i=1}^{k-1} x_i^{\alpha(i)} \prod_{j=1}^{n_1} x_{k+j-1}^{\alpha(k)\beta(j)} \prod_{i=k+1}^n x_{i+n_1-1}^{\alpha(i)} \in C.
 \end{align*}
 We show that
 \begin{equation}\label{cl:coatom1}
  \tag{Claim 1} 
  m_2 \not \in \addCo{D}.
 \end{equation}
 Since $m_1 \not \in \addCo{D}$ we know by Lemma \ref{lem:addCo} that
there are different $j_1,j_2\leq n_1$ such that
$\gcd(\beta(j_1), d(m))< d(m)$ and $\gcd(\beta(j_2), d(m))<d(m)$
and for all $i \in D$ there exists a 
$J_i \leq n_1$
such that $\gcd(\beta(J_i), P_i) = 1$.
 We have for all $i\in D$ that
 $\gcd(\alpha(k), P_i) = 1$,
 and thus we have $\gcd(\alpha(k)\beta(j_1), d(m) ) < d(m)$
 and $\gcd(\alpha(k)\beta(j_2), d(m)) < d(m)$.
 Furthermore, for all $i \in D$
 we have $\gcd(\alpha(k) \beta(J_i),P_i) = 1$.
 Since $\{\alpha(k) \beta(j) \mid j\leq n_1\}$ 
 are exponents of $m_2$,
 we get by Lemma \ref{lem:addCo} that $m_2 \not \in \addCo{D}$,
 which proves \ref{cl:coatom1}.
 
 Now we show that
 \begin{equation}\label{cl:coatom2}
  \tag{Claim 2}
  m_2 \in C'.
 \end{equation}
  Since $\alpha(1) , \ldots, \alpha(n), \beta(1), \ldots, \beta(n_1) \in \N$ we have
 for each $j\in \{1, \ldots, n+n_1 -1\}$ that the exponent of the variable $x_j$ is not equal to $0$.
 Let $i \in \{1, \ldots, l\}$.
 If $i \in D'$,
 then there is a $J_i \leq n$ with $J_i\neq k$ such that
 $\gcd(\alpha(J_i), P_i)=1$.
 If $i\in D$, then there is a 
$J_i \leq n_1$ such that $\gcd(\beta(J_i), P_i) = 1$,
 and thus $\gcd(\alpha(k)\beta(J_i), P_i) = 1$.
 Since $\{\alpha(j)\mid j\leq n, j\neq k\} \cup \{\alpha(k)\beta(j) \mid j\leq n_1\}$
 are exponents of $m_2$, we have that $m_2 \in C'$,
 which proves \ref{cl:coatom2}.
 
 By \ref{cl:coatom1} we have $m_2 \not \in \addCo{D}$ and 
 by \ref{cl:coatom2} we have $m_2 \in C'$.
 This contradicts the assumption of Case 2a.

Case 2b. 
There exists $m_1 = x_1^{\beta(1)}\cdots x_{n_1}^{\beta(n_1)} \in C'$
with $m_1 \not \in \addCo{D}$:
Since $m_1\in C'$ we have that $\beta(1), \ldots, \beta(n_1) \in \N$, and 
since $m_1 \not \in \addCo{D}$ we get
by Lemma \ref{lem:addCo} that  for
all $j_1 \leq n_1$ we have
\begin{equation}\label{prop:coatom2}
\tag{P2}
 \gcd(d'(m_1, j_1), d(m)) < d(m).
\end{equation}
By substitution of monomials we get
\begin{align*}
m_2&(x_1, \ldots, x_{n_1\cdot n}) := \\& 
m(m_1(x_1, \ldots, x_{n_1}), m_1(x_{n_1 +1 }, \ldots, x_{2 \cdot n_1}), \ldots, m_1(x_{n_1(n-1)+1}, \ldots, x_{n_1 \cdot n})) \\
&= \prod_{i=1}^n \prod_{j=1}^{n_1} x_{(i-1)\cdot n_1+j}^{\alpha(i)\beta(j)}\in C. 
\end{align*}
Since $m$ and $m_1$ lie in $C'$, we have that 
$\gcd(\alpha(1), \ldots, \alpha(n),P_1\cdots P_l) = 1 $ and $\gcd(\beta(1), \ldots, \beta(n_1),P_1\cdots P_l)= 1$.
Let 
$j'\leq n$ and $j_1' \leq n_1$ be
such that
$d(m_2) = \gcd(\{\alpha(j)\beta(j_1) \mid j \leq n, j_1 \leq n_1, (j,j_1) \neq (j',j_1')\}\cup \{P_1\cdots P_l\})$.

For all $j_1\leq n_1$ we have
$$\gcd(\{\beta(j_1)\alpha(j) \mid j\leq n\}\cup \{P_1\cdots P_l\}) = \gcd(\beta(j_1), P_1\cdots P_l),$$
since $\gcd(\alpha(1), \ldots, \alpha(n), P_1\cdots P_l) = 1$.
Hence we get for all $j_1 \leq n_1$ with $j_1 \neq j_1'$ that
$d(m_2)$ divides $\beta(j_1)$,
and thus $d(m_2)$ divides $d'(m_1, j_1')$.
Similarly, we have for all $j\leq n$ that
$$\gcd(\{\beta(j_1)\alpha(j) \mid j_1\leq n_1\}\cup \{P_1\cdots P_l\}) = \gcd(\alpha(j), P_1\cdots P_l),$$
since $\gcd(\beta(1), \ldots, \beta(n_1), P_1\cdots P_l) = 1$.
Hence we get for all $j \leq n$ with $j \neq j'$ that
$d(m_2)$ divides $\alpha(j)$,
and thus $d(m_2)$ divides $d'(m, j')$.

This means $d(m_2)= \gcd(d'(m, j'), d(m_2))= \gcd(d'(m_1, j_1'), d(m_2))$.
Since $d(m_2)=\gcd(d'(m, j'), d(m_2))$ and $d'(m,j') \leq d(m)$
we get that
\begin{equation}\label{eq:coatomfin1}
 d(m_2) = \gcd(d'(m,j'),d(m_2)) \leq d(m).
\end{equation}
By \eqref{prop:coatom2} we have $\gcd(d'(m_1,j_1'),d(m)) < d(m)$,
and since $\gcd(d'(m_1, j_1'), d(m_2))=d(m_2)$ we get that
\begin{equation}\label{eq:coatomfin2}
 \gcd(d(m_2),d(m)) = \gcd(d(m_2),d'(m_1,j_1'),d(m)) < d(m).
\end{equation}

By \eqref{eq:coatomfin1} and \eqref{eq:coatomfin2}
we get that $d(m_2) < d(m)$.
As $m_2 \in C'$, this contradicts \eqref{prop:coatom1}.

This finishes the proof.
\end{proof}

The next theorem gives a description of all coatoms of the lattice of monomial clones on $\F_q$.

\begin{thm}\label{thm:coatoms}
Let $q>2$ be a prime power, let $l\in\N$ and let $P_1<\ldots < P_l$ be the primes that divide $q-1$.
 Then the coatoms of the lattice of monomial clones are given by
 $\genMClo{\{x_1 \cdots x_{1+P_i}\}}$ for $i\leq l$,
 and by $\addCo{D}$ for each nonempty $D \subseteq \{1,\ldots, l\}$.
\end{thm}

\begin{proof}
By Theorem~\ref{thm:topIsoDivLat} we know that for $i\leq l$,
$\genMClo{\{x_1 \cdots x_{1+P_i}\}}$ is a coatom.

Now let $D \subseteq \{1,\ldots, l\}$ be nonempty.
First, we see that the monomial clone $\addCo{D}$ is different from
the monomial clones which we have described in Theorem~\ref{thm:topIsoDivLat},
since by Lemma~\ref{lem:addCo}, $\addCo{D}$ does not contain a monomial with two exponents which are relatively prime to $q-1$.
In particular, $x_1x_2\not \in \addCo{D}$, so $\addCo{D}\subset \genMClo{\{x_1x_2\}}$.

Now we show that for different nonempty sets $D, D' \subseteq \{1, \ldots, l\}$,
$\addCo{D}$ and $\addCo{D'}$ are incomparable.
First, we assume $D \not \subseteq D'$ and $D'\not \subseteq D$.
Then we have $t_1 \in D\setminus D'$ and $t_2\in D'\setminus D$.
We get by Lemma~\ref{lem:addCo} that $x_1^{P_{t_1}}x_2^{P_{t_1}}\in \addCo{D}$,
but $x_1^{P_{t_1}}x_2^{P_{t_1}}\not \in \addCo{D'}$.
On the other hand,
$x_1^{P_{t_2}}x_2^{P_{t_2}}\in \addCo{D'}$, but $x_1^{P_{t_2}}x_2^{P_{t_2}}\not\in \addCo{D}$,
and thus $\addCo{D}$ and $\addCo{D'}$ are incomparable.
Now let w.l.o.g. $D \subset D'$, and let $t\in D'\setminus D$.
Then by Lemma~\ref{lem:addCo}, $x_1^{P_t}x_2^{P_t} \not \in \addCo{D}$ and $x_1^{P_t}x_2^{P_t} \in \addCo{D'}$,
but $x_1x_2^{\prod_{i\in D}P_i}  \in \addCo{D}$ and $x_1x_2^{\prod_{i\in D}P_i} \not \in \addCo{D'}$.
Hence, $\addCo{D}$ and $\addCo{D'}$ are incomparable.

The next goal is to show that for $D \subseteq \{1,\ldots, l\}$ nonempty, $\addCo{D}$ is a coatom.
To this end,
let $C$ be a monomial clone on $\F_q$ with $\addCo{D} \subset C$.
Now we show that 
\begin{equation}\label{eq:coatomseq1}
 C = \genMClo{\{x_1x_2\}}.
\end{equation}
Let $T:= \prod_{i\in D} P_i$.
Let $m=x_1^{\alpha(1)}x_2^{\alpha(2)}\cdots x_n^{\alpha(n)} \in C \setminus \addCo{D}$ with $\alpha(1),\ldots, \alpha(n) \in \N$.
Then $n > 1$, since $\addCo{D}$ contains all monomials of width $1$.
By Lemma~\ref{lem:addCo},
$m$ has at least two exponents which are not divisible by $T$,
and for all $i \in D$ we have that there is a $j\leq n$ such that $P_i$ does not divide $\alpha(j)$.
Therefore we have $\gcd(\alpha(1), \ldots, \alpha(n), T) = 1$.

Case 1.
We assume that there is a $j\leq n$ such that $\gcd(\alpha(j), T) = 1$:
 Then there is a $j' \leq n$ with $j\neq j'$ such that $\gcd(\alpha(j'), T) < T$, since $m \not \in \addCo{D}$.
By permuting variables we assume w.l.o.g.\ that $j = 1$ and $j' = 2$.
We have that $x_1x_2^T \in \addCo{D} \subset C$, and thus $x_1x_{n+1}^T \in C$.
Now we plug $m$ into $x_1$ of $x_1x_{n+1}^T $, and get $(\prod_{i=1}^{n}x_i^{\alpha(i)})x_{n+1}^T \in C$.
Since all monomials of width $1$ are in $\addCo{D} \subset C$,
we get for 
all $t_1,t_2 \in \N$ that
$$x_1^{t_1\cdot \alpha(1)} (\prod_{i=2}^{n}x_{i}^{\alpha(i)}) x_{n+1}^{t_2\cdot T}\in C.$$
By identifying $x_{n+1}$ with $x_1$
we get 
for all $t_1, t_2 \in \N$ that
$$x_1^{t_1\cdot \alpha(1)+t_2\cdot T} (\prod_{i=2}^{n}x_{i}^{\alpha(i)}) \in C.$$
Since $\gcd(\alpha(1), T) = 1$, we have $x_1 \prod_{i=2}^n x_{i}^{\alpha(i)}\in C$ by shifting coefficients in B\'{e}zout's identity by suitable multiples of $q-1$
and closure of $C$ under equivalent monomials.
Then we get by variable identification that $x_1x_2^{\alpha(2)}x_3^{\gamma} \in C$ for some $\gamma \in \N_0$,
and thus we get by Lemma~\ref{lem:combRule}, variable permutation and identification that 
 there exists a $\gamma' \in \N_0$ such that
$x_1 x_2^{\alpha(2)} x_3^{\alpha(2)}x_{4}^{\gamma'} \in C$.
Since $\gcd(\alpha(2), T) < T$,
there is an $i \in D$
such that $P_i$ does not divide $\alpha(2)$.
We have $x_1^{P_i} x_2^{P_i} \in \addCo{D} \subset C$.
Now we get by permuting variables and by substitution of monomials 
that 
$$x_1^{P_i} x_{2}^{P_i} x_{3}^{\alpha(2)} x_{4}^{\alpha(2)}x_{5}^{\gamma'} \in C.$$
Since all monomials of width $1$ are in $\addCo{D} \subset C$,
we get for 
all $t_1,t_2 \in \N$ that
$$x_1^{t_1\cdot P_i} x_{2}^{t_1 \cdot P_i} x_{3}^{t_2\cdot \alpha(2)} x_{4}^{t_2 \cdot \alpha(2)}x_{5}^{\gamma'} \in C.$$
By identifying $x_3$ with $x_1$,
$x_4$ with $x_2$, and renaming $x_5$ to $x_3$,
we get 
for all $t_1, t_2 \in \N$ that
$$x_1^{t_1\cdot P_i + t_2 \cdot \alpha(2)} x_{2}^{t_1\cdot P_i + t_2 \cdot \alpha(2)} x_{3}^{\gamma'} \in C.$$
Since $\gcd(\alpha(2), P_i) = 1$, we have $x_1 x_2 x_3^{\gamma'} \in C$ by shifting coefficients in B\'{e}zout's identity by suitable multiples of $q-1$
and closure of $C$ under equivalent monomials.
Since $x_1^{q-1} \in C$,
we get $x_1x_2x_3^{\gamma' \cdot(q-1)} \in C$, and thus we get by Lemma \ref{lem:cutQ1} that $x_1x_2 \in C$.
Therefore $C= \genMClo{\{x_1x_2\}}$.

Case 2.
Now we assume that for all $j\leq n$ we have $\gcd(\alpha(j), T) \neq 1$:
For each $j\leq n$, we define $I(j) := \min(\{i \in \N \mid i\in D, P_i\text{ divides } \alpha(j)\})$ which is nonempty,
since there is an $i\in D$ such that $P_i$ divides $\alpha(j)$.
We have for all $j\leq n$,
that $h_j:=x_1^{P_{I(j)}}x_2^{P_{I(j)}} \in \addCo{D} \subset C$.
By substitution of monomials we get
that
\begin{align*}
 m'&:= m(h_1(x_1, x_{n+1}), h_2(x_2, x_{n+2}) \cdots, h_n(x_n,x_{2n})) \\&=
\prod_{j=1}^n x_j^{P_{I(j)}\cdot \alpha(j)} \prod_{j=1}^{n} x_{n+j}^{P_{I(j)}\cdot \alpha(j)} \in C.
\end{align*}
We have by Lemma \ref{lem:addCo} that
$x_1x_2^Tx_3^T \in \addCo{D} \subset C$, and thus $x_1x_{2n+1}^Tx_{2n+2}^T \in C$.
Now we plug $m'$ into $x_1$ of $x_1x_{2n+1}^Tx_{2n+2}^T$,
and get 
$$ \prod_{j=1}^n x_j^{ P_{I(j)}\cdot \alpha(j)} \prod_{j=1}^{n} x_{n+j}^{ P_{I(j)}\cdot \alpha(j)}x_{2n+1}^T x_{2n+2}^T \in C.$$
Since all monomials of width $1$ are in $C$,
we get for all $t_1, \ldots, t_{n+1} \in \N$
that 
$$\prod_{j=1}^n x_j^{t_j \cdot P_{I(j)}\cdot \alpha(j)} \prod_{j=1}^{n} x_{n+j}^{t_j \cdot P_{I(j)}\cdot \alpha(j)}x_{2n+1}^{t_{n+1}\cdot T} x_{2n+2}^{t_{n+1}\cdot T}  \in C.$$
By setting $x_1, \ldots, x_n, x_{2n+1}$ to $x_1$
and by setting $x_{n+1}, \ldots, x_{2n}, x_{2n+2}$ to $x_2$,
we get for all $t_1, \ldots, t_{n+1} \in \N$
that 
$$x_1^{\sum_{j=1}^n t_j \cdot P_{I(j)}\cdot \alpha(j)+t_{n+1}\cdot T} x_2^{\sum_{j=1}^n t_j \cdot P_{I(j)}\cdot \alpha(j)+t_{n+1}\cdot T} \in C.$$
We have $\gcd(\alpha(1), \ldots, \alpha(n), T) = 1$,
and for all $j\leq n$ we have 
$\gcd(\alpha(j), P_{I(j)}) = P_{I(j)}$,
and thus
$\gcd(P_{I(1)}\cdot \alpha(1), \ldots, P_{I(n)} \cdot \alpha(n), T) = 1$, since we have that $\gcd(\alpha(1), \ldots, \alpha(n), T) = 1$
and $P_{I(j)}\cdot \alpha(j)$ contains the same prime factors as $\alpha(j)$, only with different multiplicities.
Hence, by shifting coefficients in B\'{e}zout's lemma and closure of $C$ under equivalent monomials, $x_1x_2 \in C$, and thus $C= \genMClo{\{x_1x_2\}}$.

We proved \eqref{eq:coatomseq1},
and thus $\addCo{D}$ is a coatom.

Now we show that there are no other coatoms.
Let $C$ be a coatom of the lattice of monomial clones on $\F_q$ which is different from the coatoms given in the statement.
This means that for all  $\emptyset \neq D\subseteq \{1,\ldots, l\}$ we have $C \not \subseteq \addCo{D}$.
Now Lemma \ref{lem:noOthCo} yields $x_1\cdots x_q \in C$.
Since $C$ is a coatom, we get by Theorem \ref{thm:topIsoDivLat} that 
$C = \genMClo{\{x_1 \cdots x_{1+P_i}\}}$ for some $i\leq l$.
This contradicts that $C$ is a coatom which is different from the statement.
\end{proof}

All coatoms we have found in Theorem~\ref{thm:coatoms} are different and thus we obtain
the number of coatoms on $\F_q$.

\begin{cor}
 Let $q>2$ be a prime power, let $l\in\N$ be the number of different primes that divide $q-1$.
 Then there are $2^l -1  + l$ coatoms of the lattice of monomial clones on $\F_q$.
\end{cor}

\subsection{Summary}
We summarize this section in Figure \ref{fig:TopBot}.
Let $q>2$ be a prime power, let $l, m,n\in \N$ and let
$P_1< \ldots < P_l$ be all primes that divide $q-1$
and let $2 \leq s_1 <\ldots <s_m \leq q-1$ be all numbers such that for all $i\leq m$,
$s_i^p \equiv_{q-1} 1$ for some prime $p$ or $s_i s_i \equiv_{q-1} s_i$.
Let $D_1, \ldots, D_{2^l - 1}$ be all nonempty subsets of $\{1,\ldots,l\}$.

  \begin{figure}[ht!]
 \begin{center}
\begin{tikzpicture}[scale=0.55, transform shape]
  \node (max) at (0, 12) {$\genMClo{\{x_1x_2\}}$};
  \node (min) at (0, 0) {$\genMClo{\{x_1\}}$};

  \node (a1) at (-4, 2) {$\genMClo{\{x_1x_2^{q-1}\}}$};
  \node (a2) at (-1, 2) {$\genMClo{\{x_1^{s_1}\}}$};
  \node (a3) at (1, 2) {$\genMClo{\{x_1^{s_2}\}}$};
  \node (a4) at (4, 2) {$\genMClo{\{x_1^{s_m}\}}$};
  \node (a) at (2.5, 2) {$\cdots$};

  \node (b1) at (-4, 3.5){$\vdots$};
  \node (b2) at (-1, 3.5){$\vdots$};
  \node (b3) at (1, 3.5){$\vdots$};
  \node (b4) at (4, 3.5){$\vdots$};

  \node (maxId) at (-5, 7) {$\genMClo{\{x_1 \cdots x_q\}}$};
  \node (i1) at (-6, 10){$\genMClo{\{x_1 \cdots x_{1+P_1}\}}$};
  \node (i2) at (-2, 10){$\genMClo{\{x_1 \cdots x_{1+P_l}\}}$};
  \node (i) at (-4 , 10){$\cdots$};
  \node (j1) at (-5, 8){$\vdots$};
  \node (j2) at (-6, 9){$\vdots$};
  \node (j3) at (-2, 9){$\vdots$};
  \node (j4) at (-10, 9){$\mathcal{D}(q-1) \cong$ };

  \node (x1xq) at (5, 6) {$\genMClo{\{x_1^tx_2^{q-1}\mid t\in \N\}}$};
  \node (k1) at (7, 10){$\addCo{D_{2^l -1}}$};
  \node (l1) at (5, 7){$\vdots$};
  \node (l3) at (7, 9){$\vdots$};

  \node (e1) at (3,10){$\addCo{D_1}$};
  \node (e3) at (5,10){$\cdots$};
  \node (e2) at (3, 9){$\vdots$};

  \draw (min) -- (a1)
	(min) -- (a2)
	(min) -- (a3)
	(min) -- (a4)
	(e1) -- (max)
	(i1) -- (max)
	(i2) -- (max)
	(k1) -- (max);

	\draw[color=gray] (-3.8,9.5) circle [x radius=5.5cm, y radius=3cm, rotate=25];

\end{tikzpicture}\caption{The frame of the lattice of monomial clones on $\F_q$}\label{fig:TopBot}
 \end{center}
\end{figure}
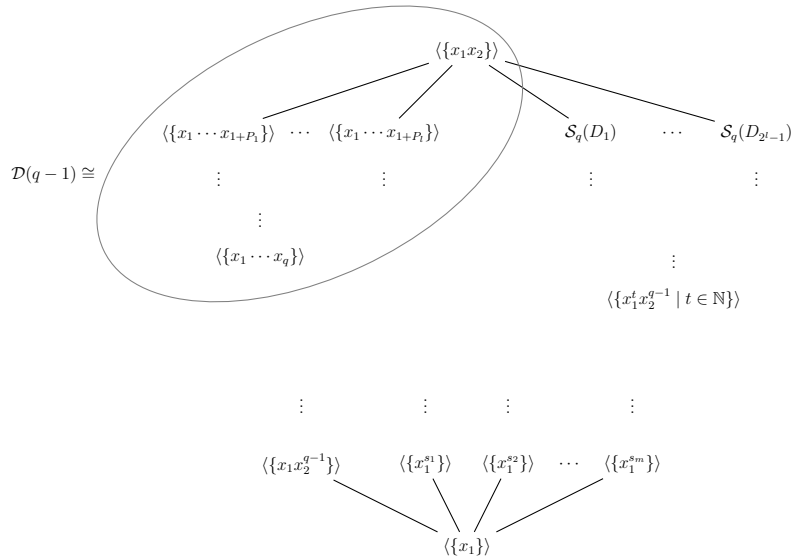

\section{Infinite chains of monomial clones}\label{sec:InfChains}

\subsection{Monomial clones are well-partially ordered}

Consider $q\in\N\setminus{\{1\}}$. Let $S\subseteq \N_0^{q-1}$ with $(0,\ldots, 0) \in S$.
We call $S$ a \emph{$q$-minor set} if for all
$(s_1,\ldots,s_{q-1}) \in S$ we have for all $t_1, \ldots, t_{q-1} \in \N_0$
with $s_i-t_i(q-1)\geq 0$ for all $i \leq q-1$
that $(s_1 -t_1(q-1), \ldots, s_{q-1}-t_{q-1}(q-1))\in S$.
By $\mathcal{T}_q$ we denote the set of all $q$-minor sets.

Let $\mathbf{b} = (b_1, \ldots, b_{q-1}) \in \{0,\ldots, q-2\}^{q-1}$ and let $S\subseteq \N_0^{q-1}$ be a $q$-minor set.
We define $M(\mathbf{b}, S)$
by
\begin{equation*}
 M(\mathbf{b}, S) := \{(t_1, \ldots, t_{q-1}) \in \N_0^{q-1} \mid (b_1, \ldots, b_{q-1}) +(q-1)\cdot(t_1, \ldots, t_{q-1}) \in S \}.
\end{equation*}

Let $\leq$ be the product order on $\N_0^{q-1}$.
We call a subset $D \subseteq \N_0^{q-1}$ \emph{downward closed}
if $d\in D$, $d' \in \N_0^{q-1}$, $d'\leq d$ implies $d' \in D$.
By the definition of a $q$-minor set,
we have that $M(\mathbf{b}, S)$ is a downward closed set for all $\mathbf{b} \in \{0,\ldots, q-2\}^{q-1}$.

\begin{lem}\label{lem:goToDownCl}
 Let $S , S' \subseteq \N_0^{q-1}$ be $q$-minor sets.
 Then $S\subseteq S'$ if and only if $M(\mathbf{b}, S) \subseteq M(\mathbf{b}, S')$ for all $\mathbf{b} \in \{0,\ldots, q-2\}^{q-1}$.
\end{lem}

\begin{proof}
 The ``only if''-direction obviously holds by the definition of $M$ as a preimage.
 For the ``if''-direction we assume that
 $S\not \subseteq S'$.
 Then there is an $s = (s_1, \ldots, s_{q-1})\in S$ such that $s\not \in S'$.
 Now let $t_1, \ldots, t_{q-1} \in \N_0$
 such that $0\leq s_i - t_i(q-1) \leq q-2$.
 We define
 $\mathbf{b} :=(s_1 - t_1(q-1), \ldots, s_{q-1} -t_{q-1}(q-1))$.
 Then we have $ (t_1, \ldots, t_{q-1})\in M(\mathbf{b}, S)$, but
 $(t_1, \ldots, t_{q-1})\not \in M(\mathbf{b}, S')$,
 because $s \not \in S'$.
This finishes the proof.
\end{proof}

Let $\mathbf{b}_1, \ldots, \mathbf{b}_{(q-1)^{(q-1)}}$ be all elements of $\{0,\ldots, q-2\}^{q-1}$
and $\mathcal{D}(\N_0^{q-1})$ be the set of all downward closed sets.
For any $q$-minor set $S\subseteq \N_0^{q-1}$,
we now define $T(S) \in (\mathcal{D}(\N_0^{q-1}))^{(q-1)^{(q-1)}}$ by
$$T(S) := (M(\mathbf{b}_1, S), \ldots, M(\mathbf{b}_{(q-1)^{(q-1)}}, S) ).$$

\begin{pro}\label{pro:monMinSetWPO}
Let $q$ be a prime power.
 $(\mathcal{T}_q,\subseteq)$ is well-partially ordered.
\end{pro}

\begin{proof}
Since the downward closed sets are the complements of the upward closed sets,
we get from \cite[Theorem 1.1]{AA:BRO}
that $(\mathcal{D}(\N_0^{q-1}),\subseteq)$ does not contain infinite antichains or infinite descending chains.
Now \cite[Example (4) p.195]{AH:FTHSIP}
yields that $((\mathcal{D}(\N_0^{q-1}))^{(q-1)^{(q-1)}}, \subseteq')$,
where $\subseteq'$ is the product order of $\subseteq$,
does not contain infinite antichains or infinite descending chains.
By Lemma~\ref{lem:goToDownCl} we have that $T$ is an order embedding into
 $((\mathcal{D}(\N_0^{q-1}))^{(q-1)^{(q-1)}}, \subseteq')$, and thus the result holds.
\end{proof}

Let now $q$ be a prime power.
Remember that $\monClones{q}$ is the set of all monomial clones on $\F_q$,
and if $C \in \monClones{q}$, then
$\overline{C}$ is the set $\{\prod_{i=1}^n x_i^{\alpha(i)}\in C \mid \forall i \leq n \colon \alpha(i) \in \{1,\ldots, q-1\}\}$.
Let $m\in \overline{C}$.
Now we define for $i\in \{1,\ldots, q-1\} $,
$\mathcal{N}(m, i)$ as the number of exponents of $m$ equal to $i$.
Let $C$ be a monomial clone.
We define
\begin{align*}
\phi(C) := \{(\mathcal{N}(m, 1), \ldots, \mathcal{N}(m, q-1))\in\N_0^{q-1}\mid m \in \overline{C}\}\cup \{(0,\ldots, 0)\}.
\end{align*}
Since $C$ is closed under identifying variables,
$\overline{a+(q-1)\cdot a} = \overline{a}$ for all $a\in \N$ 
and $(0,\ldots,0)\in \phi(C)$,
we have $\phi(C) \in \mathcal{T}_q$.

\begin{thm}\label{thm:wellPO}
Let $q$ be a prime power.
Then $\phi$ is an order embedding from $(\monClones{q},\subseteq)$ to $(\mathcal{T}_q,\subseteq)$,
and thus the lattice of monomial clones on $\F_q$ is well-partially ordered, i.e.,
there is no infinite descending chain of monomial clones on $\F_q$ and
there is no infinite antichain of monomial clones on $\F_q$.
\end{thm}

\begin{proof}
We show that $C \subseteq C'$ if and only if $\phi(C) \subseteq \phi(C')$.
The ``only if''-direction holds by the definition of $\phi$.
For the ``if''-direction let $\phi(C) \subseteq \phi(C')$.
Let w.l.o.g.\ $m = \prod_{i=1}^n x_i^{\alpha(i)} \in C$,
where $1\leq \alpha(i) \leq q-1$ for each $i\leq n$.
Then $e= (s_1, \ldots, s_{q-1})\in \phi(C)$, where $s_i = \mathcal{N}(m, i)$ for each $i\leq q-1$.
Since $\phi(C) \subseteq \phi(C')$, we have $e \in \phi(C')$.
Therefore there is an $m' \in \overline{C'} \subseteq C'$
such that $m' = \prod_{i=1}^n x_{\pi(i)}^{\alpha(i)}$ for some bijective $\pi \colon \{1,\ldots, n\} \to \{1, \ldots, n\}$.
Since $C'$ is closed under permuting variables, we have $m\in C'$.
This means $\phi$ is an order embedding from $\monClones{q}$ to $\mathcal{T}_q$.
By Proposition~\ref{pro:monMinSetWPO} we have that $(\mathcal{T}_q,\subseteq)$ is well-partially ordered, and thus the result holds.
\end{proof}

\emph{Remark}:
For the embedding from monomial clones to $q$-minor sets,
we have just used that a monomial clone is closed under composition with projections and closed under taking equivalent monomials.
Proposition 2.13 of 
\cite{S:CUA} states that
for every finite set $A$, up to term equivalence,
there are only countably many semi-affine algebras on $A$.
Theorem~\ref{thm:wellPO} does not follow directly
from this result of \'{A}. Szendrei or its proof.
On the other hand we can easily transfer the result of Theorem~\ref{thm:wellPO}
to the lattice of $0$-preserving semi-affine clones with respect to a (finite) abelian group $\mathbf{A} := (A,+,-,0)$,
by
\begin{itemize}
 \item replacing in the definition of a $q$-minor set, $(q-1)$ by the exponent of $\mathbf{A}$,
 \item and counting the number of occurrences of
 a number $r$ below the exponent among the coefficients $r_i$ of
 $f\colon A^n \to A, (a_1,\ldots, a_n) \mapsto \sum_{i=1}^n  r_i \cdot a_i$ (cf.\ the proof of \cite[Proposition 2.13]{S:CUA}).
\end{itemize}

\subsection{Infinite ascending chains}

\begin{pro}\label{pro:ascChainSquare}
  Let $q$ be a prime power such that there is a $k\in \N\setminus\{1\}$ with $k^2 \mid q-1$.
  Now let $d := \frac{q-1}{k}$.
 Let for each $i \in \N_0$,
 $f_i := \prod_{j=1}^{k\cdot i + 1} x_j^d$.
 Then $q\geq 5$ and $(\genMClo{\{f_i\}})_{i\in\N_0}$ is an infinite ascending chain of monomial clones on $\F_q$.
\end{pro}

\begin{proof}
Let $i\in\N_0$.
 We have $k\cdot d = q-1$, and thus by Lemma~\ref{lem:cutQ1} we have
 $\genMClo{\{f_i\}} \subseteq \genMClo{\{f_{i+1}\}}$.
 Now we show that this inclusion is proper.
 We have $d\neq q-1$, since $k\neq 1$, and $\overline{d^2} = \overline{(q-1) \cdot (\frac{q-1}{k^2})}=q-1$, since $k^2 \mid q-1$.
 We interpret now $f_i$ as the $(k\cdot i + 1)$-ary function which is induced by $f_i$.
  The clone generated by $f_i$ can be described as the set of all
  term functions $T^{\F_q}$
  induced by some term $T$
in the language $\{f_i\}$ over $n'$ variables where $n' \in \N$.
  We now show by induction on the length of the term $T$ that
  if $T^{\F_q}$ is not a projection, then $T^{\F_q}$ is 
  an $n$-ary operation on $\F_q$ for some $n\in\N$
  which is
  induced by a 
  monomial $x_1^{\alpha(1)} \cdots x_{n}^{\alpha(n)}$
  with $\alpha(1), \ldots, \alpha(n) \in \{0,1,\ldots,q-1\}$
  such that $\lvert \{j\leq n\mid \alpha(j) = d\} \rvert \leq k\cdot i + 1$
  and for each $j\leq n$ we have that $d$ divides $\alpha(j)$.
  If the length of the term $T$ is equal to $1$,
  then $T$ is a variable and thus induces a projection, whence the statement holds.
  Now let $T$ be of length greater than $1$.
 Then  $T = f_i(T_1, \ldots, T_{k\cdot i + 1})$
  where for $j\leq k\cdot i + 1$, $T_j$ is a shorter term than $T$.
  Let $n\in \N$ be the arity of $T^{\F_q}$.
  Now let $m = x_1^{\alpha(1)} \cdots x_{n}^{\alpha(n)}$ be the monomial 
   with $\alpha(1), \ldots, \alpha(n) \in \{0,1,\ldots,q-1\}$
  that induces $T^{\F_q}$.
  We see the following.
  Let $j\leq k\cdot i +1$.
  If $(T_j)^{\F_q}$ is a projection, then $(T_j^d)^{\F_q}(a_1, \ldots, a_n) = a_{i'}^d$ for some $i' \leq n$.
  If $(T_j)^{\F_q}$ is not a projection,
  then $(T_j^d)^{\F_q}(a_1,\ldots, a_n) = \prod_{i'\in I} a_{i'}^{q-1}$ for some $I\subseteq \{1,\ldots,n\}$,
  since by the induction hypothesis all exponents of the monomial which induces $T_j^{\F_q}$
  are of the form $t\cdot d$, and thus we have
  $t\cdot d \cdot d = t\cdot d^2$ where $\overline{t\cdot d^2} = q-1$ if $t\neq 0$.
  Since $d\mid q-1$, we now get that 
  the exponents $\alpha(1), \ldots, \alpha(n)$ of $m$
  are divisible by $d$. 
  Since $d\neq q-1$, we get that
  the number of $d$'s as exponent is at most the number of $T_j$'s which are projections.
  This means that at most $k\cdot i + 1$ many exponents of $m$ are equal to $d$.

 Now we have $\genMClo{\{f_i\}} \subset \genMClo{\{f_{i+1}\}}$,
  since $f_{i+1}$ contains $k\cdot (i+1) + 1$ many variables with $d$ as their exponent.
  Hence, $\genMClo{\{f_0\}} \subset \genMClo{\{f_1\}} \subset \genMClo{\{f_2\}} \subset \genMClo{\{f_3\}} \subset \ldots$ which finishes the proof.
\end{proof}

\begin{exa}
 In $\F_5$,
$\genMClo{\{x_1^2\}} \subset \genMClo{\{x_1^2x_2^2x_3^2\}}  \subset \genMClo{\{x_1^2x_2^2x_3^2x_4^2x_5^2\}} \subset \ldots$
 or in $\F_9$,
 $\genMClo{\{x_1^4\}} \subset \genMClo{\{x_1^4x_2^4x_3^4\}}  \subset \genMClo{\{x_1^4x_2^4x_3^4x_4^4x_5^4\}} \subset \ldots$
 are infinite ascending chains of monomial clones.
\end{exa}

\emph{Remark}:
Let $q\geq 5$ be a prime power such that there is a $k\in \N\setminus\{1\}$ with $k^2 \mid q-1$.
We find infinite ascending chains
also by the connection to semi-affine algebras:
By \cite[p. 60]{S:CUA} the clones on $\Z_{q-1}$, generated by $\{(x_1, \ldots, x_i) \mapsto \sum_{j=1}^{i} k \cdot x_j\}$ for $i \in \N$,
form an infinite ascending chain.

\begin{thm}\label{thm:squFreeAscChain}
 Let $q$ be a prime power.
 Infinite ascending chains of monomial clones on $\F_q$ exist if and only if 
 $q-1$ is not square-free.
\end{thm}

\begin{proof}
If $q-1$ is not square-free, then we get infinite ascending chains by Proposition \ref{pro:ascChainSquare}.
Now we assume that $q-1$ is square-free.
By \'{A}. Szendrei (see \cite{S:CLOFS} or \cite[p.\ 61]{S:CUA}) we know that
the number of semi-affine algebras on a finite set $A$ up to term equivalence is finite
if and only if $\lvert A \rvert$ is square-free.
Since $q-1$ is square-free we have that the number of semi-affine clones with respect to $(\Z_{q-1},+,-,0)$ is finite.
Item \eqref{en:conSem3} of Proposition~\ref{pro:conSemiAff} now yields that the number of monomial clones on $\F_q$ is finite, and thus
there do not exist infinite ascending chains.
\end{proof}

\begin{cor}\label{cor:squarefreefinlat}
 For a prime power $q$, the lattice of monomial clones on $\F_q$ is finite if and only if $q-1$ is square-free.
\end{cor}

\begin{proof}
If the lattice of monomial clones on $\F_q$ is finite, then there are no infinite ascending chains.
From Theorem~\ref{thm:squFreeAscChain} we then get that $q-1$ is square-free.
If the lattice of monomial clones on $\F_q$ is infinite, we get by the infinite version of Ramsey's theorem that
there are infinite ascending chains, or infinite descending chains or infinite antichains.
It follows from Theorem~\ref{thm:wellPO} that the lattice of monomial clones on $\F_q$ does not contain 
 infinite descending chains and does not contain infinite antichains.
Therefore the lattice of monomial clones on $\F_q$ contains an infinite ascending chain.
Now Theorem~\ref{thm:squFreeAscChain} yields that $q-1$ is not square-free.
\end{proof}

\section{Idempotent monomial clones}\label{sec:idMonCl}

Let $q$ be a prime power.
In this section we investigate idempotent monomial clones on $\F_q$, this means the interval $[\genMClo{\{x_1\}}, \genMClo{\{x_1\cdots x_q\}}]$ (cf.\ Lemma~\ref{lem:maxIdClone}).

\begin{lem}\label{lem:getOneGen}
Let $q$ be a prime power.
Let $k,l\in \N$ and let $m_1$ be a $k$-ary idempotent monomial and let $m_2$ be an $l$-ary idempotent monomial.
Now let
 $$m(x_1, \ldots, x_{k\cdot l}):= m_1(m_2(x_{1}, \ldots, x_{l}), \ldots, m_2(x_{(k-1)\cdot l +1}, \ldots, x_{l\cdot k})).$$
 Then
 $\genMClo{\{m_1, m_2\}} = \genMClo{\{m\}}$.
\end{lem}

\begin{proof}
 Let $m_1 = \prod_{j=1}^{k} x_j^{\alpha(j)}$ where $\overline{\sum_{j=1}^{k} \alpha(j)}= 1$,
 and let $m_2 = \prod_{j=1}^{l} x_j^{\beta(j)}$ where $\overline{\sum_{j=1}^{l} \beta(j)} = 1$.
  Obviously, ``$\supseteq$'' holds, since $m$ is a composition of $m_1$ and $m_2$.
 For ``$\subseteq$'' we observe the following.
 Using idempotency, we have
  \begin{align*}
  m(x_1, \ldots, x_1, x_2, \ldots, x_2, \ldots, x_{k}, \ldots, x_{k})
 = \prod_{j=1}^{k} x_{j}^{\alpha(j) \cdot (\sum_{i=1}^{l} \beta(i))} \in \genMClo{\{m\}},
 \end{align*}
 and thus $m_1=\prod_{j=1}^{k} x_{j}^{\alpha(j)} \in \genMClo{\{m\}}$,
 since $\overline{\sum_{i=1}^{l} \beta(i)} = 1$.
 On the other hand we also get
 \begin{align*}
m(x_1, \ldots, x_{l}, x_1, \ldots, x_{l}, \ldots, x_{1}, \ldots, x_{l}) 
 &= \prod_{i=1}^{k} (\prod_{j=1}^{l}x_j^{\beta(j)})^{\alpha(i)}
 = \left(\prod_{j=1}^{l}   x_j^{\beta(j)}\right)^{\sum_{i=1}^k \alpha(i)}\\
 &= \prod_{j=1}^{l} x_j^{\beta(j)\cdot  ( \sum_{i=1}^{k}\alpha(i))} \in \genMClo{\{m\}},
 \end{align*}
 and thus $m_2 = \prod_{j=1}^{l} x_j^{\beta(j)} \in \genMClo{\{m\}}$,
 since $\overline{\sum_{i=1}^{k} \alpha(i)} = 1$.
 This finishes the proof.
\end{proof}

As a consequence of Lemma~\ref{lem:getOneGen} we get the following corollary:
\begin{cor}\label{cor:finGenIdMonIsSingle}
Let $q$ be a prime power.
 A finitely generated idempotent monomial clone on $\F_q$
 is singly generated.
\end{cor}

\emph{Remark}: Corollary~\ref{cor:finGenIdMonIsSingle} does not hold in general for an arbitrary monomial clone,
as we can see from the case distinction for the field $\F_3$ in the proof of Proposition~\ref{pro:smallMonClon} where
the monomial clone $\genMClo{\{x_1^2,x_1x_2^2\}}$ cannot be generated by one element.

Now we use again the connection of monomial clones to semi-affine algebras to show that there
are only finitely many idempotent monomial clones on $\F_q$.

\begin{pro}\label{pro:idClFin}
 Let $q$ be a prime power.
 The lattice of idempotent monomial clones on $\F_q$ is finite.
\end{pro}

\begin{proof}
Let $\varphi$ be defined as in Section \ref{sec:conToSemi}.
Then $C$ is an idempotent monomial clone on $\F_q$
if and only if $\varphi(C)$ is a finite idempotent $0$-preserving semi-affine clone with respect to $(\Z_{q-1},+,-,0)$,
since $(x_1, \ldots, x_n) \mapsto \sum_{i=1}^n r(i)x_i$ is idempotent if and only if $\sum_{i=1}^n r(i) = 1$.
By Lemma~\ref{lem:smIdClone} we know that every idempotent monomial clone strictly above $\genMClo{\{x_1\}}$
contains $x_1 x_2^{q-1}$.
By Item \eqref{en:conSem1} and Item \eqref{en:conSem2} of Proposition~\ref{pro:conSemiAff}
we get that $\lvert [\genMClo{\{x_1\}}, \genMClo{\{x_1\cdots x_q\}}]\rvert =
\lvert [\varphi(\genMClo{\{x_1\}}), \varphi(\genMClo{\{x_1\cdots x_q\}})] \rvert + 1$.
The result follows now from \cite[Corollary 2.20]{S:CUA}, which states that modulo term equivalence there are only finitely many idempotent
semi-affine algebras on a fixed finite universe.
\end{proof}

\begin{cor}
Let $q$ be a prime power.
 Every idempotent monomial clone on $\F_q$ is singly generated.
\end{cor}

\begin{proof}
By Proposition~\ref{pro:idClFin} we get that there are no infinite ascending chains of monomial clones in $[\genMClo{\{x_1\}}, \genMClo{\{x_1\cdots x_q\}}]$,
and thus the result follows from Corollary~\ref{cor:finGenIdMonIsSingle}.
\end{proof}

  \section*{Acknowledgements}
The author gratefully thanks Erhard Aichinger and Stefano Fioravanti for many hours of fruitful discussions.
The author also thanks the anonymous referee for a very careful report
which led to corrections of some proofs and even of some statements 
and thus significantly improved the quality of the paper.

\providecommand{\bysame}{\leavevmode\hbox to3em{\hrulefill}\thinspace}
\providecommand{\MR}{\relax\ifhmode\unskip\space\fi MR }
\providecommand{\MRhref}[2]{
  \href{http://www.ams.org/mathscinet-getitem?mr=#1}{#2}
}
\providecommand{\href}[2]{#2}

\end{document}